\documentclass[reqno]{amsart}

\usepackage{xypic}
\input xy
\xyoption{all}
\usepackage{epsfig}
\usepackage{color}
\usepackage{amsthm}
\usepackage{amssymb}
\usepackage{amsmath}
\usepackage{amscd}
\usepackage{amsopn}

\newcommand{\labell}[1] {\label{#1}}

\numberwithin{equation}{section}
\newtheorem {Theorem}{Theorem}
\numberwithin{Theorem}{section}

\newtheorem {Lemma}[Theorem]    {Lemma}

\newtheorem {Proposition}[Theorem]{Proposition}

\theoremstyle{definition}

\theoremstyle{remark}
\newtheorem{Remark}[Theorem]{Remark}
\newtheorem{Example}[Theorem]{Example}

\expandafter\chardef\csname pre amssym.def at\endcsname=\the\catcode`\@
\catcode`\@=11
\def\undefine#1{\let#1\undefined}
\def\newsymbol#1#2#3#4#5{\let\next@\relax
 \ifnum#2=\@ne\let\next@\msafam@\else
 \ifnum#2=\tw@\let\next@\msbfam@\fi\fi
 \mathchardef#1="#3\next@#4#5}
\def\mathhexbox@#1#2#3{\relax
 \ifmmode\mathpalette{}{\m@th\mathchar"#1#2#3}%
 \else\leavevmode\hbox{$\m@th\mathchar"#1#2#3$}\fi}
\def\hexnumber@#1{\ifcase#1 0\or 1\or 2\or 3\or 4\or 5\or 6\or 7\or 8\or
 9\or A\or B\or C\or D\or E\or F\fi}

\font\teneufm=eufm10
\font\seveneufm=eufm7
\font\fiveeufm=eufm5
\newfam\eufmfam
\textfont\eufmfam=\teneufm
\scriptfont\eufmfam=\seveneufm
\scriptscriptfont\eufmfam=\fiveeufm

\catcode`\@=\csname pre amssym.def at\endcsname

\def    \eps    {\epsilon}

\newcommand{\CM}{{\mathcal M}}

\newcommand{\CS}{{\mathcal S}}

\newcommand{\const}{{\mathit const}}

\def    \C      {{\mathbb C}}
\def    \R      {{\mathbb R}}

\def    \Z      {{\mathbb Z}}

\def    \T      {{\mathbb T}}

\def    \12    {{\frac{1}{2}}}

\def    \p      {\partial}
\def    \codim  {\operatorname{codim}}

\def    \HF     {\operatorname{HF}}
\def    \HC     {\operatorname{HC}}

\def    \CF     {\operatorname{CF}}

\def    \Sp     {\operatorname{Sp}}
\def    \tSp     {\operatorname{\widetilde{Sp}}}
\def    \U     {\operatorname{U}}
\def    \sign     {\operatorname{sign}}
\def    \tPhi     {\tilde{\Phi}}
\def    \tF     {\tilde{F}}
\def    \tf     {\tilde{f}}
\def    \tPsi     {\tilde{\Psi}}
\def    \tDelta     {\tilde{\Delta}}
\def    \tH     {\tilde{H}}
\def    \tgamma     {\tilde{\gamma}}
\def    \trho     {\tilde{\rho}}
\def \inv       {\scriptscriptstyle{\mathrm{inv}}}

\def    \MUCZ  {\operatorname{\mu_{\scriptscriptstyle{CZ}}}}

\def    \ssminus        {\smallsetminus}

\begin{document}


\setlength{\smallskipamount}{6pt}
\setlength{\medskipamount}{10pt}
\setlength{\bigskipamount}{16pt}





\title[Weinstein--Moser Theorem]
{Periodic Orbits of Twisted Geodesic Flows and The Weinstein--Moser Theorem}

\author[Viktor Ginzburg]{Viktor L. Ginzburg}
\author[Ba\c sak G\"urel]{Ba\c sak Z. G\"urel}

\address{VG: Department of Mathematics, UC Santa Cruz,
Santa Cruz, CA 95064, USA}
\email{ginzburg@math.ucsc.edu}

\address{BG: Centre de recherches math\'ematiques,
Universit\'e de Montr\'eal,
P.O. Box 6128, Centre--ville Station,
Montr\'eal, QC, H3C 3J7, CANADA}
\email{gurel@crm.umontreal.ca}

\date{\today} 
\thanks{The work is partially supported by the NSF and by the faculty
research funds of the University of California, Santa Cruz.}

\bigskip

\begin{abstract}
In this paper, we establish the existence of periodic orbits of a
twisted geodesic flow on all low energy levels and in all dimensions
whenever the magnetic field form is symplectic and spherically
rational. This is a consequence of a more general theorem concerning
periodic orbits of autonomous Hamiltonian flows near Morse--Bott
non-degenerate, symplectic extrema.  Namely, we show that all energy
levels near such extrema carry periodic orbits, provided that the
ambient manifold meets certain topological requirements. This result
is a partial generalization of the Weinstein--Moser theorem. The proof
of the generalized Weinstein--Moser theorem is a combination of a
Sturm--theoretic argument and a Floer homology calculation.
\end{abstract}

\maketitle

\section{Introduction and main results}
\labell{sec:intro}

In the early 1980s, V.I. Arnold proved, as a consequence of the
Conley--Zehnder theorem, \cite{CZ1}, the existence of periodic orbits
of a twisted geodesic flow on $\T^2$ with symplectic magnetic field
for all energy levels when the metric is flat and low energy levels
for an arbitrary metric, \cite{Ar2}.  This result initiated an
extensive study of the existence problem for periodic orbits of
general twisted geodesic flows via Hamiltonian dynamical systems
methods and in the context of symplectic topology, mainly focusing on
low energy levels. (A brief and admittedly incomplete survey of some
related work is provided in Section \ref{sec:other-results}.)

In the present paper, we establish the existence of periodic orbits of
a twisted geodesic flow on all low energy levels and in all dimensions
whenever the magnetic field form is symplectic and spherically
rational. An essential point is that, in contrast with other results
of this type, we do not require any compatibility conditions on the
Hamiltonian and the magnetic field. In fact, we prove a more general
theorem concerning periodic orbits of autonomous Hamiltonian flows
near Morse--Bott non-degenerate, symplectic extrema.  Namely, we show
that all energy levels near such extrema carry periodic orbits,
provided that the ambient manifold meets certain topological
requirements.  This result is a (partial) generalization of the
Weinstein--Moser theorem, \cite{Mo:orbits,We:orbits}, asserting that a
certain number of distinct periodic orbits exist on every energy level
near a non-degenerate extremum. The proof of the generalized
Weinstein--Moser theorem is a combination of a Sturm--theoretic
argument utilizing convexity of the Hamiltonian in the direction
normal to the critical submanifold and of a Floer--homological
calculation that guarantees ``dense existence'' of periodic orbits
with certain index. The existence of periodic orbits for a twisted
geodesic flow with symplectic magnetic field is then an immediate
consequence of the generalized Weinstein--Moser theorem.

\subsection{The generalized Weinstein--Moser theorem}
\labell{sec:WM}

Throughout the paper, $M$ will stand for a closed symplectic
submanifold of a symplectic manifold $(P,\omega)$. We denote by
$[\omega]$ the cohomology class of $\omega$ and by $c_1(TP)$ the first
Chern class of $P$ equipped with an almost complex structure
compatible with $\omega$. The integrals of these classes over a
2-cycle $u$ will be denoted by $\left<\omega,u\right>$ and,
respectively, $\left<c_1(TP),u\right>$. Recall also that $P$ is said
to be \emph{spherically rational} if the integrals
$\left<\omega,u\right>$ over all $u\in\pi_2(P)$ are commensurate,
i.e., $\lambda_0=\inf\{|\left<\omega,u\right>|\mid u\in\pi_2(P)\}>0$
or, equivalently, $\left<\omega,\pi_2(P)\right>$ is a discrete
subgroup of $\R$.

The key result of the paper is

\begin{Theorem}[Generalized Weinstein--Moser theorem]
\labell{thm:main}

Let $K\colon P\to \R$ be a smooth function on a symplectic manifold
$(P,\omega)$, which attains its minimum $K=0$ along a closed
symplectic submanifold $M\subset P$. Assume in addition that the
critical set $M$ is Morse--Bott non-degenerate and one of the
following cohomological conditions is satisfied:
\begin{itemize}
\item[(i)] $M$ is spherically rational and $c_1(TP)=0$, or
\item[(ii)] $c_1(TP)=\lambda[\omega]$ for some $\lambda\neq 0$.
\end{itemize}
Then for every sufficiently small $r^2>0$ the level $K=r^2$ carries a
contractible in $P$ periodic orbit of the Hamiltonian flow of $K$ with
period bounded from above by a constant independent of $r$.
\end{Theorem}

When $M$ is a point, Theorem \ref{thm:main} turns into the
Weinstein--Moser theorem (see \cite{We:orbits} and \cite{Mo:orbits})
on the existence of periodic orbits near a non-degenerate extremum,
albeit without the lower bound $1+\dim P/2$ on the number of periodic
orbits.

\begin{Remark}
The assertion of the theorem is local and concerns only a neighborhood
of $M$ in $P$. Hence, in (i) and (ii), we can replace $c_1(TP)$ by
$c_1(TP|_M)=c_1(TM)+c_1(TM^\perp)$ and $[\omega]$ by
$[\omega|_M]$. Also note that in (ii) we do not require $\lambda$ to
be positive, i.e., $M$ need not be monotone. (However, this condition
does imply that $M$ is spherically rational.) We also emphasize that
we do need conditions (i) and (ii) in their entirety -- the weaker
requirements $c_1(TP)|_{\pi_2(P)}=0$ or
$c_1(TP)|_{\pi_2(P)}=\lambda[\omega]|_{\pi_2(P)}$, common in
symplectic topology, are not sufficient for the proof.

Although conditions (i) and (ii) enter our argument in an essential
way, their role is probably technical (see Section \ref{sec:role}),
and one may expect the assertion of the theorem to hold without any
cohomological restrictions on $P$. For instance, this is the case
whenever $\codim M=2$; see \cite{Gi96}.  Furthermore, when $\codim
M\geq 2$ the theorem holds without (i) and (ii), provided that the
normal direction Hessian $d^2_M K$ and $\omega$ meet a certain
geometrical compatibility requirement; \cite{GK1,GK2,Ke1}. On the
other hand, the condition that the extremum $M$ is Morse--Bott
non-degenerate is essential; see \cite{GG}.
\end{Remark}

\subsection{Periodic orbits of twisted geodesic flows}
\labell{sec:magn}

Let $M$ be a closed Riemannian manifold and let $\sigma$ be a closed
2-form on $M$. Equip $T^*M$ with the twisted symplectic structure
$\omega=\omega_0+\pi^*\sigma$, where $\omega_0$ is the standard
symplectic form on $T^*M$ and $\pi\colon T^*M\to M$ is the natural
projection.  Denote by $K$ the standard kinetic energy Hamiltonian on
$T^*M$ corresponding to a Riemannian metric on $M$.  The Hamiltonian
flow of $K$ on $T^*M$ describes the motion of a charge on $M$ in the
\emph{magnetic field} $\sigma$ and is sometimes referred to as a
magnetic or \emph{twisted geodesic flow}; see, e.g., \cite{Gi:newton}
and references therein for more details. Clearly, $c_1(T(T^*M))=0$,
for $T^*M$ admits a Lagrangian distribution (e.g., formed by spaces
tangent to the fibers of $\pi$), and $M$ is a Morse--Bott
non-degenerate minimum of $K$. Furthermore, $M$ is a symplectic
submanifold of $T^*M$ when the form $\sigma$ symplectic. Hence, as an
immediate application of case (i) of Theorem \ref{thm:main}, we obtain

\begin{Theorem}
\labell{thm:magn}

Assume that $\sigma$ is symplectic and spherically rational. Then for
every sufficiently small $r^2>0$ the level $K=r^2$ carries a
contractible in $T^*M$ periodic orbit of the twisted geodesic flow
with period bounded from above by a constant independent of~$r$.
\end{Theorem}

\begin{Remark}
The proof of Theorem \ref{thm:main} is particularly transparent when
$P$ is \emph{geometrically bounded} and \emph{symplectically
aspherical} (i.e., $\omega|_{\pi_2(P)}=0=c_1(TP)|_{\pi_2(P)}$). This
particular case is treated in Section \ref{sec:proof-sa}, preceding
the proof of the general case.  The twisted cotangent bundle
$(T^*M,\omega)$ is geometrically bounded; see \cite{AL,CGK,Lu}.
Furthermore, $(T^*M,\omega)$ is symplectically aspherical if and only
if $(M,\sigma)$ is weakly exact (i.e., $\sigma|_{\pi_2(M)}=0$).
\end{Remark}

Note also that, as the example of the horocycle flow shows, a twisted
geodesic flow with symplectic magnetic field need not have periodic
orbits on all energy levels; see, e.g., \cite{CMP,Gi:newton} for a
detailed discussion of this example and of the resulting transition in
the dynamics from low to high energy levels. Similar examples also
exist for twisted geodesic flows in dimensions greater than two,
\cite[Section 4]{Gi:bayarea}.

\subsection{Related results}
\labell{sec:other-results}

To the best of the authors' knowledge, the existence problem for
periodic orbits of a charge in a magnetic field was first addressed by
V.I. Arnold in the early 1980s; \cite{Ar2,Ko}. Namely, V.I. Arnold
established the existence of at least three periodic orbits of a
twisted geodesic flow on $M=\T^2$ with symplectic magnetic field for
all energy levels when the metric is flat and low energy levels for an
arbitrary metric. (It is still unknown if the second of these results
can be extended to all energy levels.) Since then the question has
been extensively investigated. It was interpreted (for a symplectic
magnetic field) as a particular case of the generalized
Weinstein--Moser theorem in \cite{Ke1}. Referring the reader to
\cite{Gi:newton,Gi:arnold,Gi:alan} for a detailed review and further
references, we mention here only some of the results most relevant to
Theorems \ref{thm:main} and \ref{thm:magn}.

The problems of \emph{almost existence} and \emph{dense existence} of
periodic orbits concern the existence of periodic orbits on almost all
energy levels and, respectively, on a dense set of levels.  In the
setting of the generalized Weinstein--Moser theorem or of twisted
geodesic flows, these problems are studied for low energy levels in,
e.g., \cite{CGK,Co,CIPP,FS,GG,Gu,Ke,Ma,Lu,Lu2,Schl}, following the
original work of Hofer and Zehnder and of Struwe,
\cite{FHW,HZ:V,HZ:cap,HZ,St}.  In particular, almost existence for
periodic orbits near a symplectic extremum is established in
\cite{Lu2} under no restrictions on the ambient manifold $P$. When $P$
is geometrically bounded and (stably) strongly semi-positive, almost
existence is proved for almost all low energy levels in \cite{Gu}
under the assumption that $\omega|_M$ does not vanish at any point,
and in \cite{Schl} when $M$ has middle-dimension and $\omega|_M\neq
0$. These results do not require the extremum $M$ to be Morse--Bott
non-degenerate. Very strong almost existence results (not restricted
to low energy levels) for twisted geodesic flows with exact magnetic
fields and also for more general Lagrangian systems are obtained in
\cite{Co,CIPP}.  The dense or almost existence results established in
\cite{CGK,GG,Ke} follow from Theorem \ref{thm:main}.  However, the
proof of Theorem \ref{thm:main} relies on the almost existence theorem
from \cite{GG} or, more precisely, on the underlying Floer homological
calculation.

As is pointed out in Section \ref{sec:WM}, in the setting of the
generalized Weinstein--Moser theorem without requirements (i) and
(ii), every low energy level carries a periodic orbit whenever $\codim
M=2$ or provided that the normal direction Hessian $d^2_M K$ and
$\omega$ meet certain geometrical compatibility conditions, which are
automatically satisfied when $\codim M=2$ or $M$ is a point; see
\cite{Gi:FA,Gi96,GK1,GK2,Ke1,Mo:orbits,We:orbits} and references
therein. Moreover, under these conditions, non-trivial lower bounds on
the number of distinct periodic orbits have also been obtained.  The
question of existence of periodic orbits of twisted geodesic flows on
(low) energy levels for magnetic fields on surfaces is studied in,
e.g., \cite{No,NT,Ta1,Ta2} in the context of Morse--Novikov theory;
see also \cite{Co,CIPP,CMP,Gi:arnold} for further references. (In
general, this approach requires no non-degeneracy condition on the
magnetic field.)  For twisted geodesic flows on surfaces with exact
magnetic fields, existence of periodic orbits on all energy levels is
proved in \cite{CMP}.

\subsection{Infinitely many periodic orbits} The multiplicity results
from \cite{Ar2,Gi:FA,Gi96,GK1,GK2,Ke1} rely (implicitly in some
instances) on the count of ``short'' periodic orbits of the
Hamiltonian flow on $K=r^2$. The resulting lower bounds on the number
of periodic orbits can be viewed as a ``crossing-over'' between the
Weinstein--Moser type lower bounds in the normal direction to $M$ and
the Arnold conjecture type lower bounds along $M$. This approach
encounters serious technical difficulties unless $\omega$ and $d^2_MK$
meet some geometrical compatibility requirements, for otherwise even
identifying the class of short orbits is problematic.

However, looking at the question from the perspective of the Conley
conjecture (see \cite{FrHa,Gi:conley,Hi,SZ}) rather than of the Arnold
conjecture, one can expect every low level of $K$ to carry infinitely
many periodic orbits (not necessarily short), provided that $\dim
M\geq 2$ and $M$ is symplectically aspherical. An indication that this
may indeed be the case is given by

\begin{Proposition}
\labell{prop:inf-many}

Assume that $M$ is symplectically aspherical and not a point, and
$\codim M=2$ and the normal bundle to $M$ in $P$ is trivial. Then
every level $K=r^2$, where $r>0$ is sufficiently small, carries
infinitely many distinct, contractible in $P$ periodic orbits of $K$.
\end{Proposition}

This proposition does not rely on Theorem \ref{thm:main} and is an
immediate consequence of the results of \cite{Ar2,Gi:FA} and the
Conley conjecture; see \cite{Gi:conley} and also \cite{FH,Hi,SZ}. For
the sake of completeness, a detailed argument is given in Section
\ref{sec:pf-inf-many}. In a similar vein, in the setting of Theorem
\ref{thm:magn} with $M=\T^2$ and $K$ arising from a flat metric, the
level $K=r^2$ carries infinitely many periodic orbits for every (not
necessarily small) $r>0$.

\subsection{Outline of the proof of Theorem \ref{thm:main} 
and the organization of the paper}

The proof of Theorem \ref{thm:main} hinges on an interplay of two
counterparts: a version of the Sturm comparison theorem and a Floer
homological calculation. Namely, on the one hand, a Floer homological
calculation along the lines of \cite{GG} guarantees that almost all
low energy levels of $K$ carry periodic orbits with Conley--Zehnder
index depending only on the dimensions of $P$ and $M$. On the other
hand, since the levels of $K$ are fiber-wise convex in a tubular
neighborhood of $M$, a Sturm theoretic argument ensures that periodic
orbits with large period must also have large index. (Strictly
speaking, the orbits in question are degenerate and the
Conley--Zehnder index is not defined. Hence, we work with the
Salamon--Zehnder invariant $\Delta$, \cite{SZ}, but the Robin--Salamon
index, \cite{RS}, could be utilized as well.) Thus, the orbits
detected by Floer homology have period \emph{a priori} bounded from
above and the existence of periodic orbits on all levels follows from
the Arzela--Ascoli theorem.

The paper is organized as follows. In Section \ref{sec:SZ}, we recall
the definition and basic properties of the Salamon--Zehnder invariant
$\Delta$ and also prove a version of the Sturm comparison theorem
giving a lower bound for the growth of $\Delta$ in linear systems with
positive definite Hamiltonians. This lower bound is extended to
periodic orbits of $K$ near $M$ in Propositions \ref{prop:growth} and
\ref{prop:growth'} of Section \ref{sec:growth}, providing the
Sturm--theoretic counterpart of the proof of Theorem
\ref{thm:main}. In Section \ref{sec:proof-sa}, we prove Theorem
\ref{thm:main} under the additional assumptions that $P$ is
geometrically bounded and symplectically aspherical. In this case,
clearly illustrating the interplay between Sturm theory and Floer
homology, we can directly make use of a Floer homological calculation
from \cite{GG}.  Turning to the general case, we define in Section
\ref{sec:FN} a version of filtered Floer (or rather Floer--Novikov)
homology of compactly supported Hamiltonians on open manifolds.  The
relevant part of the calculation from \cite{GG} is extended to the
general setting in Section \ref{sec:general}. The proof of Theorem
\ref{thm:main} is completed in Section \ref{sec:gen-pf} where we also
discuss some other approaches to the problem.  Proposition
\ref{prop:inf-many} is proved in Section \ref{sec:pf-inf-many}.

\subsection*{Acknowledgements}
The authors are deeply grateful to Michael Entov, Ely Kerman, and
Leonid Polterovich for valuable discussions, remarks and suggestions.

\section{The Salamon--Zehnder invariant $\Delta$}
\labell{sec:SZ}

In this section we briefly review the properties of the invariant
$\Delta$, a continuous version of the Conley--Zehnder index introduced
in \cite{SZ}, used in the proof of Theorem \ref{thm:main}.

\subsection{Linear algebra}
\labell{sec:la}

Let $(V,\omega)$ be a symplectic vector space. Throughout this paper
we denote the group of linear symplectic transformations of $V$ by
$\Sp(V,\omega)$ or simply $\Sp(V)$ when the form $\omega$ is clear
from the context. Moreover, if $V$ is also equipped with a complex
structure $J$ we will use the notation $\U(V,\omega,J)$ or just
$\U(V)$ for the group of unitary transformations, i.e.,
transformations preserving $J$ and $\omega$. For $A\in \U(V)$ we
denote by $\det_{\C}A\in S^1$ the complex determinant of~$A$.

Salamon and Zehnder, \cite{SZ}, proved that there exists a unique
collection of continuous maps $\rho\colon \Sp(V,\omega)\to S^1\subset
\C$, where $(V,\omega)$ ranges through all finite-dimensional
symplectic vector spaces, with the following properties:

\begin{itemize}

\item
For any $A\in \Sp(V,\omega)$ and any linear isomorphism $B\colon W\to
V$, we have $\rho(B^{-1}AB)=\rho(A)$. (Note that $B^{-1}AB\in
\Sp(W,B^*\omega)$.) In particular, $\rho$ is conjugation invariant on
$\Sp(V,\omega)$.

\item
Whenever $A_1\in \Sp(V_1,\omega_1)$ and $A_2\in \Sp(V_2,\omega_2)$, we
have $\rho(A_1\times A_2)=\rho(A_1)\rho(A_2)$, where $A_1\times A_2$
is viewed as a symplectic transformation of $(V_1\times
V_2,\omega_1\times \omega_2)$.

\item For $A\in \U(V,\omega,J)$, we have $\rho(A)=\det_{\C}A$.

\item
For $A$ without eigenvalues on the unit circle, $\rho(A)=\pm 1$.
\end{itemize}

Note that $\rho(A)$ is completely determined by the eigenvalues of $A$
together with a certain ``ordering'' of eigenvalues, and in fact only
the eigenvalues of $A$ on the unit circle matter. It is also worth
emphasizing that $\rho$ is not smooth on $\Sp(V)$. Furthermore,
although in general $\rho(AB)\neq \rho(A)\rho(B)$, we have
$$
\rho(A^k)=\rho(A)^k
$$
for all $k\in\Z$. In particular, $\rho(A^{-1})=\overline{\rho(A)}$.

\subsection{The Salamon--Zehnder quasi-morphism $\Delta$}
\labell{sec:Delta}

\subsubsection{Definition and basic properties}
In this section we recall the definition and basic properties of the
\emph{Salamon--Zehnder invariant} $\Delta$ following closely
\cite{SZ}. Let $\Phi\colon [a,\,b]\to \Sp(V)$ be a continuous
path. Pick a continuous function $\lambda \colon [a,\,b]\to \R$ such
that $\rho(\Phi(t))=e^{2\pi i \lambda(t)}$ and set
$$
\Delta(\Phi)=\frac{\lambda(b)-\lambda(a)}{\pi}\in\R.
$$
It is clear that $\Delta(\Phi)$ is independent of the choice of
$\lambda$ and that geometrically $\Delta(\Phi)$ measures the total
angle swept by $\rho(\Phi(t))$ as $t$ varies from $a$ to $b$. Note
also that we do not require $\Phi(a)$ to be the identity
transformation.

As an immediate consequence of the definition, $\Delta(\Phi)$ is an
invariant of homotopy of $\Phi$ with fixed end-points. In particular,
$\Delta$ gives rise to a continuous map $\tSp(V)\to \R$, where
$\tSp(V)$ is the universal covering of $\Sp(V)$. Furthermore,
$\Delta(\Phi)$ is an invariant of (orientation preserving)
reparametrizations of $\Phi$. On the other hand, let $\Phi^{\inv}$ be
the path $\Phi$ traversed in the opposite direction. Then
$$
\Delta(\Phi^{\inv})=\Delta(\Phi^{-1})=-\Delta(\Phi).
$$
Finally, $\Delta$ is additive with respect to concatenation of
paths. More explicitly, assume that $a<c<b$. Then, in obvious
notation,
$$
\Delta(\Phi|_{[a,\,b]})=\Delta(\Phi|_{[a,\,c]})+\Delta(\Phi|_{[c,\,b]}).
$$

From conjugation invariance of $\rho$, we see that
$\Delta(\Psi^{-1}\Phi\Psi)=\Delta(\Phi)$ for any two continuous paths
$\Phi$ and $\Psi$ in $\Sp(V)$. Moreover, when $B\colon W\to V$ is a
symplectic transformation,
\begin{equation}
\labell{eq:conj}
\Delta(B^{-1}\Phi B)=\Delta (\Phi).
\end{equation}

Finally, assume that $\Phi(0)=I$ and $\Phi(T)-I$ is
non-degenerate. (Here $\Phi\colon [0,\, T]\to \Sp(V)$.) Then the
\emph{Conley--Zehnder index} $\MUCZ(\Phi)$ is defined (see \cite{CZ2}
and also, e.g., \cite{Sa,SZ}) and, as is shown in \cite{SZ},
\begin{equation}
\labell{eq:CZ-SZ1}
|\MUCZ(\Phi)-\Delta(\Phi)|\leq \dim V/2.
\end{equation}

We refer the reader to \cite{SZ} for proofs of these facts and for a
more detailed discussion of the invariant $\Delta$.

\subsubsection{The quasi-morphism property} One additional property
of $\Delta$ important for the proof of Theorem \ref{thm:main} is that
$\Delta\colon \tSp(V)\to \R$ is a \emph{quasi-morphism}, i.e., for any
two elements $\Phi$ and $\Psi$ in $\tSp(V)$, we have
\begin{equation}
\labell{eq:qm1}
|\Delta(\Psi\Phi)-\Delta(\Psi)-\Delta(\Phi)|\leq C,
\end{equation}
where the constant $C\geq 0$ is independent of $\Psi$ and $\Phi$.

To simplify the notation, throughout the rest of this section we will
denote by $C$ a positive constant depending only on $\dim V$ -- as is
the case in \eqref{eq:qm1}. However, $C$ may assume different values
in different formulas.

With this convention in mind, \eqref{eq:qm1} is easily seen to be
equivalent to that
\begin{equation}
\labell{eq:qm2}
|\Delta(A\Phi)-\Delta(\Phi)|\leq C
\end{equation}
for any continuous path $\Phi$ in $\Sp(V)$, not necessarily
originating at the identity, and for any $A\in\Sp(V)$.

The quasi-morphism property \eqref{eq:qm1} is well known to hold for
several other maps $\tSp(V)\to\R$ which are similar to $\Delta$ (see
\cite{BG}) and can be established for $\Delta$ in a number of ways as
a consequence of the quasi-morphism property for one of these maps.

For instance, recall that every $A\in \Sp(V)$ can be uniquely
represented as a product $A=QU$, where $U$ is unitary (with respect to
a fixed, compatible with $\omega$ complex structure) and $Q$ is
symmetric and positive definite. (This is the so-called \emph{polar
decomposition}.) Set $\trho(A)=\det_{\C}U$ and define $\tDelta$ in the
same way as $\Delta$, but with $\trho$ in place of $\rho$.  (In
contrast with $\rho$ and $\Delta$, the maps $\trho$ and $\tDelta$
depend on the choice of complex structure.)  It is known that the map
$\tDelta\colon \tSp(V)\to \R$ is a quasi-morphism; see \cite{Du} and
also \cite{BG} for further references. Furthermore, as is shown in
\cite[Section C-2]{BG},
$\Delta(\Phi)=\lim_{k\to\infty}\tDelta(\Phi^k)/k$ for
$\Phi\in\tSp(V)$. Now it is easy to see that \eqref{eq:qm1} holds for
$\Delta$ since it holds for $\tDelta$.

\begin{Remark}
Alternatively, to prove \eqref{eq:qm1}, one can first show that
$|\tDelta-\Delta|\leq C$ on $\tSp(V)$ and then use again the fact that
$\tDelta$ is a quasi-morphism.  (This argument was communicated to us
by M. Entov and L. Polterovich, \cite{Po}.) In fact, once the
inequality $|\tDelta-\Delta|\leq C$ is established, it is not hard to
prove directly that both maps $\Delta$ and $\tDelta$ are
quasi-morphisms by using the polar decomposition and ``alternating''
between these two maps.  The only step which is, perhaps, not
immediate is that \eqref{eq:qm2} holds for $\Delta$ when $A$ and
$\Phi$ are both symmetric and positive definite. This, however,
follows from the elementary fact that in this case the eigenvalues of
$A\Phi(t)$ are real for all $t$ (even though $A\Phi(t)$ is not
necessarily symmetric), and hence $\Delta(A\Phi)=\Delta(\Phi)=0$.
\end{Remark}

\begin{Remark}
It is worth mentioning that any of Maslov type quasi-morphisms on
$\tSp(V)$ (see, e.g., \cite{BG,EP,RS,SZ}) can be used in the proof of
Theorem \ref{thm:main}. The only features of a quasi-morphism
essential for the argument are the normalization (behavior on $\U(V)$)
and the Sturm comparison theorem (Proposition \ref{prop:comparison}
below).  The latter obviously holds for any of these quasi-morphisms,
once it is established for one, for the difference between any two of
such quasi-morphisms is bounded.  The properties that set $\Delta$
apart from other quasi-morphisms are that $\Delta$ is continuous and
conjugation invariant and homogeneous (i.e.,
$\Delta(\Phi^k)=k\Delta(\Phi)$; see \cite{SZ}). These facts, although
used in the proof for the sake of simplicity, are not really crucial
for the argument.
\end{Remark}

\subsection{Sturm comparison theorem}
\labell{sec:sturm}

A time-dependent, quadratic Hamiltonian $H(t)$ on $(V,\omega)$
generates a linear time-dependent flow $\Phi_H(t)\in\Sp(V)$ via
the Hamilton equation. Once $V$ is identified with $\R^{2n}=\C^n$,
this equation takes the form
$$
\dot{\Phi}_H=JH(t)\Phi_H(t),
$$
where $J$ is the standard complex structure. We say that $H_1\geq H_0$
when $H_1-H_0$ is positive semi-definite, i.e., $H_1-H_0$ is a
non-negative function on $V$. Likewise, we write $H_1-H_0>0$ if
$H_1-H_0$ is positive definite.

\begin{Proposition}[Sturm Comparison Theorem]
\labell{prop:comparison}
Assume that $H_1\geq H_0$ for all $t$. Then
$$
\Delta(\Phi_{H_1})\geq \Delta(\Phi_{H_0}) -C
$$
as functions of $t$.
\end{Proposition}

This result is yet another version of the comparison theorem in
(symplectic) Sturm theory, similar to those established in, e.g.,
\cite{Ar,Bo,Ed}. The proposition can be easily verified by combining
the construction of the generalized Maslov index, \cite{RS}, with the
Arnold comparison theorem, \cite{Ar}, and utilizing
\eqref{eq:CZ-SZ1}. For the sake of completeness, we give a detailed
proof.

\begin{proof}
Due to continuity of $\Delta$, by perturbing $H_1$ and $H_0$ if
necessary, we may assume without loss of generality that $H_1-H_0>0$
for all $t$. Furthermore, by the quasi-morphism property
\eqref{eq:qm2}, we may also assume that $\Phi_{H_0}(0)=\Phi_{H_1}(0)$.

Set $H_s=(1-s)H_0+sH_1$ and let $\Phi_s(t)$ stand for the flow of
$H_s$ with the initial condition $\Phi_s(0)$ independent of $s$. Thus
\begin{equation}
\labell{eq:c1}
\dot{\Phi}_s=JH_s\Phi_s.
\end{equation}

Fix $T>0$. The path $\Phi_1(t)$ with $t\in [0,\, T]$ is homotopic to
the concatenation of $\Phi_0(t)$ and the path $\Psi(s)=\Phi_s(T)$,
$s\in [0,\,1]$. Hence, it suffices to show that
\begin{equation}
\labell{eq:c2}
\Delta(\Psi)\geq -C.
\end{equation}

Denote by $K_s(t)$ the quadratic Hamiltonian generating the family
$s\mapsto \Phi_s(t)$ for a fixed time $t\in [0,\, T]$.  To establish
\eqref{eq:c2}, let us first show that $K_s(T)>0$ for all $s\in
[0,\,1]$.  Using continuity of $\Delta$ as above, we may assume
without loss of generality that $K_s(t)$ degenerates only for a finite
collection of points $0=t_0< t_1<\ldots< t_k<T$. It is well known that
the positive inertia index of $K_s(t)$ increases as $t$ goes through
$t_i$ provided that the restriction of $\dot{K}_s(t_i)$ to $\ker
K_s(t_i)$ is positive definite; see e.g., \cite{Ar}.  Linearizing the
Hamilton equation \eqref{eq:c1} with respect to $s$, we obtain by a
simple calculation that
$$
\dot{K}_s=\dot{H}_s+\{K_s,H_s\},
$$
where $\{K_s,H_s\}=H_sJK_s-K_sJH_s$ (the Poisson bracket). Note that
$\{K_s,H_s\}(x)=-2\left<K_sx,JH_sx\right>$. Hence, $\{K_s,H_s\}(t_i)$
vanishes on $\ker K_s(t_i)$. Furthermore, $\dot{H}_s=H_1-H_0>0$ on $V$
and, as a consequence, $\dot{K}_s(t_i)$ is positive definite on $\ker
K_s(t_i)$.  Finally, $K_s(0)=0$, for $\Phi_s(0)$ is independent of
$s$, and we conclude that $K_s(t)> 0$ for all $s\in [0,\,1]$ and all
$t\in (0,\, T]$ and, in particular, for $t=T$.

Returning to the proof of \eqref{eq:c2}, set
$\tPsi(s)=\Psi(s)\Psi(0)^{-1}$.  This family is again generated by
$K_s(T)$, but now the initial condition is $\tPsi(0)=I$. Due to the
quasi-morphism property \eqref{eq:qm2}, it suffices to prove that
$\Delta(\tPsi)\geq -C$. We will show that $\Delta(\tPsi)\geq 0$.  As
above, by continuity, we may assume that $I-\tPsi(s)$ degenerates only
for a finite collection of points $0=s_0< s_1<\ldots< s_l<1$.  (In
particular, $I-\tPsi(1)$ is non-degenerate.) Then $\MUCZ(\tPsi)$ is
defined and, as is proved in \cite{RS},
$$
\MUCZ(\tPsi)=\frac{1}{2}\sign(K_0(T))+\sum_i
\sign \left(K_{s_i}(T)|_{V_i}\right),
$$
where $V_i=\ker(I-\tPsi_{s_i}(T))$ and $\sign$ denotes the signature
of a quadratic form. Since, $K_s(T)>0$ for all $s$, we see that
$\MUCZ(\tPsi)\geq n$ and, by \eqref{eq:CZ-SZ1}, $\Delta(\tPsi)\geq
0$. This completes the proof of \eqref{eq:c2} and the proof of the
proposition.
\end{proof}

\begin{Example}
\labell{ex:1}

Let $H(t)$ be a quadratic Hamiltonian on $\R^{2n}$ such that
$H(t)(X)\geq \alpha \| X\|^2$ for all $t$, where $\|X\|$ stands for
the standard Euclidean norm of $X\in\R^{2n}$ and $\alpha$ is a
constant. Then, for all $t$,
$$
\Delta(\Phi_H)\geq 2n\alpha \cdot t-C.
$$
More generally, let $H(t)$ be a quadratic Hamiltonian on
$\R^{2n_1}\times\R^{2n_2}$ such that $H(t)((X,Y))\geq \alpha \|
X\|^2-\beta\|Y\|^2$ for all $t$, where $X\in\R^{2n_1}$ and
$Y\in\R^{2n_2}$ and $\alpha$ and $\beta$ are constants.  Then
$$
\Delta(\Phi_H)\geq 2(n_1\alpha-n_2\beta)t-C.
$$
These inequalities readily follow from Proposition
\ref{prop:comparison} by a direct calculation.
\end{Example}

\subsection{The Salamon--Zehnder invariant for integral curves}

\labell{sec:SZ-Hamiltonian}
\subsubsection{Definitions}
\labell{sec:integral-curves}

Let $\gamma\colon [0,\,T]\to P$ be an integral curve of the
Hamiltonian flow $\varphi^t_H$ of a time-dependent Hamiltonian $H=H_t$
on a symplectic manifold $P$. Let also $\xi$ be a symplectic
trivialization of $TP$ along $\gamma$, i.e., $\xi(t)$ is a symplectic
basis in $T_{\gamma(t)}P$ depending smoothly or continuously on
$t$. The trivialization $\xi$ gives rise to a symplectic
identification of the tangent spaces $T_{\gamma(t)}P$ with
$T_{\gamma(0)}P$, and hence the linearization of $\varphi^t_H$ along
$\gamma$ can be viewed as a family $\Phi(t)\in\Sp(T_{\gamma(0)}P)$. We
set $\Delta_\xi(\gamma):=\Delta(\Phi)$. This is the Salamon--Zehnder
invariant of $\gamma$ with respect to $\xi$. Clearly,
$\Delta_\xi(\gamma)$ depends on $\xi$.

Assume now that $\gamma$ is a contractible $T$-periodic orbit of $H$.
Recall that a \emph{capping} of $\gamma$ is an extension of $\gamma$
to a map $v\colon D^2\to P$. A capping gives rise to a symplectic
trivialization of $TP$ along $v$ and hence along $\gamma$, unique up
to homotopy, and we denote by $\Delta_v(\gamma)$ the Salamon--Zehnder
invariant of $\gamma$ evaluated with respect to this
trivialization. Note that $\Delta_v(\gamma)$ is determined entirely by
the homotopy class of $v$ and it is well known that adding a sphere
$w\in \pi_2(P)$ to $v$ results in the Salamon--Zehnder invariant
changing by $-2\int_w c_1(TP)$. In particular,
$\Delta(\gamma):=\Delta_v(\gamma)$ is independent of $v$ whenever
$c_1(TP)|_{\pi_2(P)}=0$.

When $\gamma$ is \emph{non-degenerate}, i.e., $d\varphi_H^T\colon
T_{\gamma(0)}P\to T_{\gamma(0)}P$ does not have one as an eigenvalue,
the Conley--Zehnder index $\MUCZ(\gamma)$ is defined as $\MUCZ(\Phi)$
in the same way as $\Delta(\gamma)$ by using a trivialization along
$\gamma$; see \cite{CZ2,Sa,SZ}. Then inequality \eqref{eq:CZ-SZ1}
relating $\Delta$ and $\MUCZ$ turns into
\begin{equation}
\labell{eq:CZ-SZ}
|\MUCZ(\gamma)-\Delta(\gamma)|\leq \dim P/2.
\end{equation}
Note that in general $\MUCZ(\gamma)$ depends on the choice of
trivialization along $\gamma$. Thus, in \eqref{eq:CZ-SZ} we assumed
that both invariants are taken with respect to the same
trivialization, e.g., with respect to the same capping, unless
$c_1(TP)|_{\pi_2(P)}=0$ and the choice of capping is immaterial for
either invariant; see, e.g., \cite{Sa}. When the choice of capping $v$
is essential, we will use the notation $\Delta_v(\gamma)$ and
$\MUCZ(\gamma,v)$.

\begin{Example}
\labell{ex:2}

Let $K\colon \R^{2n}\to \R$ be a convex autonomous Hamiltonian such
that $d^2 K\geq \alpha\cdot I$ at all points, where $\alpha$ is a
constant. Then, as is easy to see from Example \ref{ex:1},
$\Delta(\gamma)\geq 2n\alpha\cdot T-C$ for any integral curve
$\gamma\colon [0,\, T]\to \R^{2n}$.  Note that here $\Delta(\gamma)$
is evaluated with respect to the standard Euclidean trivialization and
we are not assuming that the curve $\gamma$ is closed.
\end{Example}

\subsubsection{Change of the Hamiltonian}
Consider two autonomous Hamiltonians $H$ and $K$ on a symplectic
manifold $P$ such that $H$ is an increasing function of $K$, i.e.,
$H=f\circ K$, where $f\colon\R\to\R$ is an increasing function. Let
$\gamma$ be a periodic orbit of $K$ lying on an energy level,
which is regular for both $K$ and $H$. 
Then $\gamma$ can also be viewed, up to a change of time, as a
periodic orbit of $H$. Fixing a trivialization of $TP$ along $\gamma$,
we have the Salamon--Zehnder invariants, $\Delta(\gamma,K)$ and
$\Delta(\gamma,H)$ of $\gamma$ defined for the flows of $K$ and
$H$. The following result, used in the proof of Theorem
\ref{thm:main}, is nearly obvious:

\begin{Lemma}
\labell{lemma:change}
Under the above assumptions, $\Delta(\gamma,K)=\Delta(\gamma,H)$.
\end{Lemma}

\begin{proof} Set $H_s=(1-s)K+sH$, where $s\in [0,\, 1]$. These Hamiltonians
are functions of $H_0=K$ and the level containing $\gamma$ is regular
for each $H_s$.  Furthermore, after multiplying $K$ and $H$ by
positive constants, we may assume that $\gamma$ has period equal to
one for all $H_s$. Denote by $\Phi_s(t)$ the linearization of the flow
$\varphi^t_{H_s}$ of $H_s$ along $\gamma$ interpreted, using the
trivialization, as a path in $\Sp(T_zP)$, where
$z=\gamma(0)$. Clearly,
$$
\Delta(\gamma,K)=\Delta(\Phi_0)\text{ and }\Delta(\gamma,H)=\Delta(\Phi_1).
$$
The path $\Phi_1(t)$ is homotopic to the concatenation of $\Phi_0(t)$
and the path $\Psi(s)=\Phi_s(1)$. Hence,
$$
\Delta(\Phi_1)=\Delta(\Phi_0)+\Delta(\Psi),
$$
and it is sufficient to show that $\Delta(\Psi)=0$. To this end, we
will prove that all maps $\Psi(s)=(d\varphi^1_{H_s})_z\colon T_zP\to
T_zP$ have the same eigenvalues.

Note that for all $s$ the maps $\Psi(s)$ are symplectic and preserve
the hyperplane $E$ tangent to the energy level through $z$. The
eigenvalues of $\Psi(s)$ are those of $\Psi(s)|_E$ and the eigenvalue
one corresponding to the normal direction to $E$. Furthermore, all
maps $\Psi$ also preserve the one-dimensional space $E^\omega$ spanned
by $\gamma'(0)$ and are equal to the identity on this space. The
quotient $E/E^\omega$ can be identified with the space normal to
$\gamma'(0)$ in $E$ and the map $\bar{\Psi}(s)\colon E/E^\omega\to
E/E^\omega$ induced by $\Psi(s)|_E$ is the linearized return map along
$\gamma$ in the energy level containing $\gamma$. 
Thus, this map is independent of $s$. As a consequence, the
maps $\Psi(s)|_E$, and hence $\Psi(s)$, have the same eigenvalues for
all $s\in [0,\, 1]$.
\end{proof}

\section{Sturm comparison theorems for periodic orbits near
Morse--Bott non-degenerate symplectic extrema}
\labell{sec:growth}

\subsection{Growth of $\Delta$} Let, as in Theorem \ref{thm:main},
$K\colon P\to \R$ be an autonomous Hamiltonian attaining its
Morse-Bott non-degenerate minimum $K=0$ along a closed symplectic
submanifold $M\subset P$. The key to the proof of Theorem
\ref{thm:main} is the following result, generalizing Example
\ref{ex:2}, which is essentially a version of the Sturm comparison
theorem for $K$:

\begin{Proposition}
\labell{prop:growth}
Assume that $c_1(TP)=0$. Then there exist constants $a>0$ and $c$ and
$r_0>0$ such that, whenever $0<r<r_0$,
\begin{equation}
\labell{eq:growth}
\Delta(\gamma)\geq a\cdot T-c
\end{equation}
for every contractible $T$-periodic orbit $\gamma$ of $K$ on the level
$K=r^2$.
\end{Proposition}

Along with this proposition, we also establish a lower bound on
$\Delta(\gamma)$ that holds without the assumption that $c_1(TP)=0$.
Fix a closed 2-form $\sigma$ with $[\sigma]=c_1(TP)$. For instance, we
can take as $\sigma$ the Chern--Weil form representing $c_1$ with
respect to a Hermitian connection on $TP$. In the notation of Section
\ref{sec:integral-curves}, we have

\begin{Proposition}
\labell{prop:growth'}
There exist constants $a>0$ and $c$ and $r_0>0$ such that, whenever
$0<r<r_0$,
\begin{equation}
\labell{eq:growth'}
\Delta_v(\gamma)\geq a\cdot T-c-2\int_v\sigma
\end{equation}
for every contractible $T$-periodic orbit $\gamma$ of $K$ on the level
$K=r^2$ with capping $v$.
\end{Proposition}

\subsection{Proof of Propositions \ref{prop:growth} and \ref{prop:growth'}}
The idea of the proof is that the fiber contribution to $\Delta(\gamma)$ is of
order $T$ and positive, while the base contribution is of
order $r\cdot T$.  It will be convenient to prove a superficially more
general form of \eqref{eq:growth} and \eqref{eq:growth'}. Namely, we
will show that
\begin{equation}
\labell{eq:growth2}
\Delta(\gamma)\geq (a-b\cdot r) T-c
\end{equation}
and
\begin{equation}
\labell{eq:growth2'}
\Delta_v(\gamma)\geq (a-b\cdot r) T-c-2\int_v\sigma
\end{equation}
for some constants $a>0$ and $b$ and $c$, when $r>0$ is small. This
implies \eqref{eq:growth} and \eqref{eq:growth'} with perhaps a
slightly smaller value of $a$.

Throughout the rest of this section we adopt the following notational
convention: in all expressions $\const$ stands for a constant which is
independent of $r$ and $\gamma$ and $T$, once $r$ is sufficiently
small. The value of this constant (immaterial for the proof) is
allowed to vary from one formula to another. A similar convention is
also applied to the constants $a>0$ and $b$ and $c$.

\subsubsection{Particular case: an integral curve in a Darboux chart}
\labell{sec:chart}
Before turning to the general case, let us prove \eqref{eq:growth2}
for an integral curve $\gamma$ of $K$ contained in a Darboux
chart. Let $U\subset M$ be a contractible Darboux chart. The inclusion
$U\hookrightarrow M$ can be extended to a symplectic embedding of an
open set $U\times V\hookrightarrow P$, where $V$ is a ball (centered
at the origin) in a symplectic vector space and $U\times V$ carries
the product symplectic structure. In what follows, we identify
$U\times V$ with its image in $P$ and $U$ with $U\times 0$. Note that
then $T_{(x,0)} (x\times V)$, where $x\in U$, is the symplectic
orthogonal complement $(T_x M)^\omega$ to $T_x M$.

Let $\gamma\colon [0,\, T]\to U\times V$ be an integral curve of the
flow of $K$ on an energy level $K=r^2$. We emphasize that at this
stage we do not require $\gamma$ to be closed, but we do require it to
be entirely contained in $U\times V$. The coordinate system in
$U\times V$ gives rise to a symplectic trivialization of $TP$ along
$\gamma$ and we denote by $\Delta(\gamma)$ the Salamon--Zehnder
invariant of the linearized flow along $\gamma$ with respect to this
trivialization; see Section \ref{sec:integral-curves}.

Next we claim that \emph{\eqref{eq:growth2} holds for such an integral
curve $\gamma$ with all constants independent of $\gamma$}.

Indeed, the linearized flow of $K$ along $\gamma$ is given by the
quadratic Hamiltonian equal to the Hessian $d^2K_{\gamma(t)}$
evaluated with respect to the coordinate system. On the other hand,
since $d^2 K$ is positive definite in the direction normal to the
critical manifold $M$, we have
\begin{equation}
\labell{eq:hessian-lower-bound}
d^2 K_{(x,y)}(X,Y)\geq a \|Y\|^2-b\cdot r \| X\|^2.
\end{equation}
Here $(x,y)\in U\times V$ and $X\in T_xU$ and $Y\in T_yV$ and
$r^2=K(x,y)$.  Note that the constants $a>0$ and $b$ depend on $K$ and
the coordinate chart $U\times V$, but not on $\gamma$ and $r$. The
lower bound \eqref{eq:growth2} (with values of $a$ and $b$ different
from those in \eqref{eq:hessian-lower-bound}) follows now from the
comparison theorem (Proposition \ref{prop:comparison}) and Example
\ref{ex:1}; cf.\ Example \ref{ex:2}.

\subsubsection{Length estimate} Fix an almost complex structure $J$ on $P$
compatible with $\omega$ and such that $M$ is an almost complex
submanifold of $P$, i.e., $J(TM)= TM$. The pair $J$ and $\omega$ gives
rise to a Hermitian metric on the complex vector bundle $TP\to P$. We
denote by $l(\gamma)$ the length of a smooth curve $\gamma$ in $P$
with respect to this metric. Furthermore, there exists a unique
Hermitian connection on $TP$, i.e., a unique connection such that
parallel transport preserves the metric and $J$, and hence,
$\omega$. (Note that, unless $J$ is integrable, this connection is
different from the Levi--Civita connection.)

Let $\gamma\colon [0,\, T]\to P$ be an integral curve of $K$ (not
necessarily closed) on the level $K=r^2$. Then, since $M$ is a
critical manifold of $K$, we have
\begin{equation}
\labell{eq:length}
l(\gamma)\leq \const \cdot r \cdot T.
\end{equation}

As the first application of \eqref{eq:length}, observe that
Proposition \ref{prop:growth} is a consequence of Proposition
\ref{prop:growth'}, i.e., \eqref{eq:growth2'} implies
\eqref{eq:growth2}. Indeed, assume that $c_1(TP)=0$, i.e.,
$\sigma=d\alpha$ for some 1-form $\alpha$ on $P$.  Then, by Stokes'
formula and \eqref{eq:length},
$$
\left|\int_v\sigma\right|=\left|\int_\gamma\alpha\right|\leq
\const\cdot \|\alpha\|_{C^0} \cdot r \cdot T,
$$
which, combined with \eqref{eq:growth2'}, implies \eqref{eq:growth2}.

Before proceeding with a detailed proof of \eqref{eq:growth2'}, let us
briefly outline the argument.  We will cover a closed $T$-periodic
orbit $\gamma$ of $K$ on the level $K=r^2$ by a finite collection of
Darboux charts. The required number $N$ of charts  is of order
$l(\gamma)\sim r\cdot T$. Within every chart, as was proved in Section
\ref{sec:chart}, we have a lower bound on $\Delta$ with respect to the
Euclidean trivialization. Combined, these trivializations can be
viewed as an approximation to a Hermitian-parallel trivialization
$\xi$ along $\gamma\colon [0,\,T]\to P$. (We do not assume that
$\xi(0)=\xi(T)$.) Furthermore, within every chart the discrepancy
between Salamon--Zehnder invariants for the two trivializations
(Euclidean and Hermitian-parallel) is bounded by a constant
independent of $\gamma$ and $r$. As a consequence, the difference
between $\Delta_\xi(\gamma)$ and the total Salamon--Zehnder invariant
for Euclidean chart-wise trivializations is of order $N\sim r\cdot T$,
and we conclude that \eqref{eq:growth2} holds for
$\Delta_\xi(\gamma)$. Finally, by the Gauss--Bonnet theorem, the
effect of replacing $\xi$ by a trivialization associated with a
capping is captured by the integral term in \eqref{eq:growth2'}.

\subsubsection{Auxiliary structure: a Darboux family}
\labell{sec:darboux}
To introduce a Darboux family in $P$ along $M$, let us first set some
notation. Denote by $B_x(\delta)\subset T_x M$ and
$B_x^\perp(\delta^\perp)\subset (T_xM)^\omega$ the balls of radii
$\delta>0$ and $\delta^\perp>0$, respectively, centered at the origin
and equipped with the symplectic structures inherited from $T_xM$ and
$(T_xM)^\omega$.

The first component of the Darboux family is a symplectic
tubular neighborhood $\pi\colon W\to M$. This is an ordinary tubular
neighborhood of $M$, i.e., an identification of a neighborhood $W$ of
$M$ in $P$ with a neighborhood of the zero section in $(TM)^\omega =
TM^\perp$ formed by the fiber-wise balls $B_x^\perp(\delta^\perp)$,
such that the diffeomorphisms between the fibers $V_x=\pi^{-1}(x)$ and
the balls $B_x^\perp(\delta^\perp)$ preserve the symplectic
structure. In particular, we obtain a family of symplectic embeddings
$B_x^\perp(\delta^\perp)\to P$ sending the origin to $x$ and depending
smoothly on $x$. The linearization of the map
$B_x^\perp(\delta^\perp)\to V_x$ at $x$ is the inclusion
$(T_xM)^\omega\hookrightarrow T_xP$.

The second component is a Darboux family in $M$. This is a family of
symplectic embeddings $T_xP\supset B_x(\delta)\to M$ depending
smoothly on $x\in M$, sending the origin $0\in T_x M$ to $x$, and
having the identity linearization at $0\in T_x M$. It is easy to see
that such a Darboux family exists provided that $\delta>0$ is
sufficiently small; see \cite{We84}. We denote the images of this
embedding by $U_x \subset M$.

Now we extend each pair of symplectic embeddings
$B_x^\perp(\delta^\perp)\to P$ and $B_x(\delta)\to M$ to a symplectic
embedding $T_xP\supset B_x(\delta)\times B_x^\perp(\delta^\perp)\to P$,
which is again required to depend smoothly on $x\in M$. The resulting
maps will be called a \emph{Darboux family (in $P$ along $M$)}.

Let $W_x$ stand for the image of the embedding
$B_x(\delta)\times B_x^\perp(\delta^\perp)\to P$. Note that $W_x$ is
naturally symplectomorphic to $U_x\times V_x$ with the split
symplectic structure and the tangent space to $y\times V_x$ is
$(T_yM)^\omega$ for every $y\in U_x$. We also denote by $\pi_x\colon
W_x\to U_x$ the projection to the first factor.  (At this point it is
worth emphasizing that in general $\pi$ and $\pi_x$ do not agree on
$W_x$ although $\pi(V_x)=x=\pi_x(V_x)$.)  Whenever the values of radii
$\delta$ and $\delta^\perp$ are essential, we will use the notation
$U_x(\delta)$ and $V_x(\delta^\perp)$ and $W_x(\delta,\delta^\perp)$
and $W(\delta^\perp)$.  Henceforth, we fix a Darboux family with some
$\delta_0>0$ and $\delta_0^\perp>0$ and consider only Darboux families
obtained by restricting the fixed one to smaller balls.

Let us now state a few simple properties of Darboux families, which
are used in the rest of the proof. These properties require $\delta>0$
and $\delta^\perp>0$ to be sufficiently small. However, once this is
the case, all constants involved are independent of $\delta$ and
$\delta^\perp$.

\begin{itemize}
\item[(DF1)]
The Euclidean metric on $W_x$, arising from the Darboux diffeomorphism
of $W_x$ with an open subset of $T_xP$, is equivalent to the
restriction of the Hermitian metric to $W_x$. Moreover, the constants
involved can be taken independent of $x$.
\end{itemize}

As a consequence of this obvious observation we need not distinguish
between the Hermitian and Euclidean metric on $W_x$ in (DF2) and (DF3)
below.

\begin{itemize}
\item[(DF2)]
The inequality \eqref{eq:hessian-lower-bound} holds in each chart
$W_x$ with some constants $a>0$ and $b$ independent of $x$.

\item[(DF3)]
The difference between Euclidean and Hermitian parallel transports
along any short curve contained in $W_x$ is small for all $x\in
M$. More specifically, denote by $\Pi_\eta^E$ and $\Pi^H_\eta$ the
Euclidean and Hermitian parallel transports $T_{\eta(0)}P\to
T_{\eta(1)}P$ along a curve $\eta\colon [0,\,1]\to W_x$. For any
$\eps>0$ there exists $l_0$, depending on $\eps$ but not on $\delta$
and $\delta^\perp$, such that for any $x\in M$ and any curve $\eta$ in
$W_x$ with $l(\eta)\leq l_0$, the symplectic transformation
$(\Pi^H_\eta)^{-1} \Pi_\eta^E $ lies in the $\eps$-neighborhood of the
identity in $\Sp(T_{\eta(0)}P)$.
\end{itemize}

The property (DF2) is a consequence of the fact that the linearization
of a Darboux map $B_x\times B_x^\perp \to W_x$ at the origin is the
identity map on $T_xP$. Assertion (DF3) is established by the
standard argument.

Now we fix a small $\eps>0$ and $\delta>0$ and $\delta^\perp>0$ such
that (DF1) and (DF2) hold and the distance from $V_x(\delta^\perp/2)$
to the boundary of $W_x=W_x(\delta,\delta^\perp)$ is smaller than
$l_0(\eps)$. This is possible since $l_0$ is independent of $\delta$
and $\delta^\perp$.

\begin{Remark}
\labell{rmk:epsilon}
In fact, $\eps>0$ need not be particularly small. It suffices to
ensure that the value of the Salamon--Zehnder invariant $\Delta$ on
any path in the $\eps$-neighborhood of the identity is bounded by a
constant independent of the path. This is always the case when the
neighborhood is simply connected (and has compact closure).
\end{Remark}

\subsubsection{Proof of \eqref{eq:growth2'}}
Let $r>0$ be so small that the level $K=r^2$ is entirely contained in
the tubular neighborhood $W(\delta^\perp/2)$. Then this level is also
contained in the union of the charts $W_x(\delta,\delta^\perp/2)$ and
hence in the union of the charts $W_x$. Let $\gamma\colon [0,\,T]\to
P$ be a $T$-periodic orbit of $K$ on the level.

Fix a unitary frame $\xi(0)$ in $T_{\gamma(0)}P$ and extend this frame
to a Hermitian trivialization $\xi$ of $TP$ along the path $\gamma$ by
applying Hermitian parallel transport to $\xi(0)$. Note that the
resulting trivialization need not be a genuine trivialization along
$\gamma$ viewed as a closed curve: $\xi(0)\neq \xi(T)$. Nonetheless,
the Salamon--Zehnder invariant $\Delta_\xi(\gamma)$ of $\gamma$ with
respect to $\xi$ is obviously defined.  Namely, recall from Section
\ref{sec:integral-curves} that using $\xi$ we can view the linearized
flow along $\gamma$ as a family $\Phi(t)\in \Sp(T_{\gamma(0)}P)$. Then
$\Delta_\xi(\gamma)=\Delta(\Phi)$.  Our first objective is to show
that \eqref{eq:growth2} holds for $\Delta_\xi(\gamma)$, i.e.,
\begin{equation}
\labell{eq:growth3}
\Delta_\xi(\gamma)\geq (a-b\cdot r) T-c,
\end{equation}
where the constants $a>0$ and $b$ and $c$ are independent of $r$ and
$T$ and $\gamma$.

To this end, consider the partition of $I=[0,\,T]$ into intervals
$I_j=[t_{j-1},\, t_j]$ with $j=1,\ldots, N$ by points
$$
0=t_0<t_1<\ldots<t_{N-1}< t_N=T
$$
such that the length of $\gamma_j=\gamma|_{I_j}$ is exactly
$l_0$. (The last segment $\gamma_N$ may have length smaller than
$l_0$.) It is essential for what follows that, by \eqref{eq:length},
\begin{equation}
\labell{eq:N}
N\leq 1+\const \cdot r\cdot T.
\end{equation}
(Note that, in contrast with the curves $\gamma_j$, the intervals
$I_j$ are not necessarily short: the average length of $I_j$ is
$T/N\sim 1/r$.)  Let $\tau_j$ be the middle point of $I_j$, i.e.,
$\tau_j=(t_{j-1}+t_j)/2$, and $z_j=\gamma(\tau_j)$ and
$x_j=\pi(z_j)$. Due to our choice of $r$, we have $z_j\in
V_{x_j}(\delta^\perp/2)$, and, by the choice of $\delta$ and
$\delta^\perp$, the path $\gamma_j$ lies entirely in $W_{x_j}$.  We
denote by $\Phi|_{I_j}$ the restriction of the family $\Phi(t)$ to
$I_j$. Thus,
\begin{equation}
\labell{eq:sum}
\Delta_\xi(\gamma)=\sum \Delta(\Phi|_{I_j}).
\end{equation}

We bound $\Delta(\Phi|_{I_j})$ from below in a few steps. First,
consider the family $\Phi_j(t)\in \Sp(T_{z_j}P)$ parametrized by $t\in
I_j$ and obtained from the linearized flow of $K$ along $\gamma_j$ by
identifying $T_{\gamma(t)}P$ with $T_{z_j}P$ via Hermitian parallel
transport. It is easy to see that
$$
\Phi_j(t)=\Pi_j\Phi(t)\Phi(\tau_j)^{-1}\Pi_j^{-1},
$$
where $\Pi_j\colon T_{z_0}P\to T_{z_j}P$ is the Hermitian parallel
transport along $\gamma$. By conjugation invariance of $\Delta$ (see
\eqref{eq:conj}) and the quasi-morphism property \eqref{eq:qm2},
\begin{equation}
\labell{eq:Phi2}
\Delta(\Phi|_{I_j})\geq \Delta(\Phi_j)-\const,
\end{equation}
where the constant depends only on $\dim P$. Furthermore, let
$\Psi_j(t)$ be defined similarly to $\Phi_j(t)$, but this time making
use of Euclidean parallel transport in $W_{z_j}$. Clearly,
$$
\Psi_j(t)=A_j(t)\Phi_j(t),
$$
where $A_j(t)\in \Sp(T_{z_j}P)$ measures the difference between the
Hermitian and Euclidean parallel transports along $\gamma_j$. Since
$l(\gamma_j)\leq l_0$, we infer from (DF3) that $A_j(t)$ lies in the
$\eps$-neighborhood of the identity and thus $\Delta(A_j)\leq \const$,
where the constant is independent of $j$ and $\gamma$ and $r$; see
Remark \ref{rmk:epsilon}. Due to the quasi-morphism property
\eqref{eq:qm1} of $\Delta$, we have
\begin{equation}
\labell{eq:Phi3}
\Delta(\Phi_j)\geq \Delta(\Psi_j)-\Delta(A_j)-\const\geq \Delta(\Psi_j)-\const.
\end{equation}
By (DF2), the argument from Section \ref{sec:chart} applies to
$\Psi_j$, and hence
\begin{equation}
\labell{eq:Phi4}
\Delta(\Psi_j)\geq (a-b\cdot r)(t_j-t_{j-1})-\const.
\end{equation}
Combining \eqref{eq:Phi2}--\eqref{eq:Phi4}, we see that
\begin{equation}
\labell{eq:Phi5}
\Delta(\Phi|_{I_j})\geq (a-b\cdot r)(t_j-t_{j-1})-\const,
\end{equation}
where all constants are independent of $\gamma$ and $r$ and the chart,
and $a>0$.

Finally, adding up inequalities \eqref{eq:Phi5} for all $j=1,\ldots,N$
and using \eqref{eq:sum}, we obtain
$$
\Delta(\Phi)\geq (a-b\cdot r)T-\const\cdot N,
$$
which in conjunction with \eqref{eq:N} implies \eqref{eq:growth3}.

To finish the proof of \eqref{eq:growth2'}, fix a capping $v$ of
$\gamma$ and let $\zeta$ be a Hermitian trivialization of $TP$ along
$\gamma$ associated with $v$. Identifying the spaces $T_{\gamma(t)}P$
via $\zeta$, we can view the linearized flow of $K$ along $\gamma$ as
a family $\tPhi(t)\in \Sp(T_{z_0}P)$, $t\in I$. By definition,
$\Delta_v(\gamma)=\Delta(\tPhi)$. Furthermore, without loss of
generality we may assume that $\zeta(0)=\xi(0)$ and then
$$
\tPhi(t)=B(t)\Phi(t).
$$
Here the transformations $B(t)\in \U(T_{z_0}P)$ send the frame
$\zeta(0)$ to the frame $\xi(t)$, where the latter is regarded as a
frame in $T_{z_0}P$ by means of $\zeta$.  Due to again the
quasi-morphism property,
$$
\Delta(\tPhi)\geq \Delta(\Phi)+\Delta(B)-\const.
$$
Since the transformations $B(t)$ are unitary,
$\rho(B(t))=\det_{\C}B(t)$, and $\Delta(B)$ is the ``total rotation''
of $\det_{\C}^2 B$. Hence, by the Gauss--Bonnet theorem,
$$
\Delta(B)=-2\int_v\sigma,
$$
where $\sigma$ is the Chern--Weil form representing $c_1(TP)$.
Combined with \eqref{eq:growth3}, this concludes the proof of
\eqref{eq:growth2'} and the proof of Propositions \ref{prop:growth}
and \ref{prop:growth'}.

\section{Particular case: $P$ is geometrically bounded and symplectically
aspherical}
\labell{sec:proof-sa}

To set the stage for the proof of the general case, in this section we
establish Theorem \ref{thm:main} under the additional assumptions that
$P$ is geometrically bounded and symplectically aspherical (i.e.,
$\omega|_{\pi_2(P)}=0=c_1(TP)|_{\pi_2(P)}$). We refer the reader to,
e.g., \cite{AL,CGK,GG} for the definition and a detailed discussion of
geometrically bounded manifolds. Here we only mention that among such
manifolds are all closed symplectic manifolds as well as their
covering spaces, manifolds that are convex at infinity (e.g., $\R^{2n}$,
cotangent bundles, and symplectic Stein manifolds) and also twisted
cotangent bundles.

\subsection{Conventions}
\labell{sec:conventions}

Throughout the rest of the paper we adopt the following conventions
and notation.  Let $\gamma\colon S^1_T\to P$, where $S^1_T=\R/T\Z$, be
a contractible loop with capping $v$.  The action of a $T$-periodic
Hamiltonian $H$ on $(\gamma,v)$ is defined by
$$
A_H(\gamma,v)=-\int_v\omega+\int_{S^1_T} H_t(\gamma(t))\,dt,
$$
where $H_t=H(t,\cdot)$. When $\omega|_{\pi_2(P)}=0$, the action
$A_H(\gamma,v)$ is independent of the choice of $v$ and we will use
the notation $A_H(\gamma)$.

All Hamiltonians considered below are assumed to be one-periodic in
time or autonomous. In the former case, we always require $T$ to be an
integer; in the latter case, $T$ can be an arbitrary real number.

The least action principle asserts that the critical points of $A_H$
on the space of all (capped) contractible loops $\gamma\colon S^1_T\to
P$ are exactly (capped) contractible $T$-periodic orbits of the
time-dependent Hamiltonian flow $\varphi_H^t$ of $H$. The Hamiltonian
vector field $X_H$ of $H$, generating this flow, is given by
$i_{X_H}\omega=-dH$.  The Salamon--Zehnder invariant
$\Delta_v(\gamma)$ of a $T$-periodic orbit $\gamma$ with capping $v$
and the Conley--Zehnder index $\MUCZ(\gamma,v)$, when $\gamma$ is
non-degenerate, are defined as in Section~\ref{sec:integral-curves}
using the linearized flow $d\varphi_H^t$ and a trivialization
associated with $v$.

At this point it is important to emphasize that our present
conventions differ from the conventions from, e.g., \cite{Sa},
utilized implicitly in Sections \ref{sec:SZ-Hamiltonian} and
\ref{sec:growth}. For instance, the Hamiltonian vector field $X_H$
defined as above is negative of the Hamiltonian vector field in
\cite{Sa}. As a consequence of this sign change, the values of
$\Delta_v(\gamma)$ and $\MUCZ(\gamma,v)$ also change sign.  (In
other words, from now on the Salamon--Zehnder invariant of a linear
flow with positive definite Hamiltonian is negative; equivalently,
$\MUCZ$ is normalized so that $\MUCZ(\gamma)=n$ when $\gamma$ is a
non-degenerate maximum of an autonomous Hamiltonian with small
Hessian.)  In particular, the value of $\Delta$ in
Propositions \ref{eq:growth} and \ref{eq:growth'} must in what follows
be replaced by $-\Delta$. This change of normalization should not lead
to confusion, for the correct sign is always clear from
the context, and it will enable us to conveniently eliminate a number
of negative signs in the statements of intermediate results.

\subsection{Floer homological counterpart}
\labell{sec:GG1}
The proof of the theorem uses two major ingredients. One is the Sturm
comparison theorem for $K$ proved in Section \ref{sec:growth}. The
other is a calculation of the filtered Floer homology for a suitably
reparametrized flow of $K$.

Let, as in Section \ref{sec:growth}, $K\colon P\to \R$ be an
autonomous Hamiltonian attaining its Morse-Bott non-degenerate minimum
$K=0$ along a closed symplectic submanifold $M\subset P$.  Pick
sufficiently small $r>0$ and $\eps>0$ with, say, $\eps<\eps_0=r^2/10$.
Let $H\colon [r^2-\eps,\, r^2+\eps]\to [0,\infty)$ be a smooth
decreasing function such that
\begin{itemize}
\item $H\equiv \max H$ near $r^2-\eps$ and $H\equiv 0$ near $r^2+\eps$.
\end{itemize}
Consider now the Hamiltonian equal to $H\circ K$ within the shell
bounded by the levels $K=r^2-\eps$ and $K= r^2+\eps$ and extended to
the entire manifold $P$ as a locally constant function. Abusing
notation, we denote the resulting Hamiltonian by $H$ again. Clearly,
$\min H=0$ on $P$ and the maximum, $\max H$, is attained on the entire
domain $K\leq r^2-\eps$.

\begin{Proposition}[\cite{GG}]
\labell{prop:floer}

Assume that $P$ is geometrically bounded and symplectically aspherical
and that $r>0$ is sufficiently small. Then, once $\max H\geq C(r)$
where $C(r)\to 0$ as $r\to 0$, we have
$$
\HF^{(a,\,b)}_{n_0}(H)\neq 0
$$
for $n_0=1+(\codim M-\dim M)/2$ and some interval $(a,\,b)$
with $a>\max H$ and $b<\max H+C(r)$.
\end{Proposition}

Here $\HF^{(a,b)}_*(H)$ stands for the filtered Floer homology of $H$
for the interval $(a,b)$. We refer the reader to Floer's papers
\cite{F:Morse,F:grad,F:c-l,F:witten,F:hs}, to, e.g.,
\cite{BPS,HS,SZ,schwarz}, or to \cite{HZ,McSa,Sa} for further
references and introductory accounts of the construction of
(Hamiltonian) Floer and Floer--Novikov homology. Filtered Floer
homology for geometrically bounded manifolds are discussed in detail
in, e.g., \cite{CGK,GG,Gu} and \cite{Gi:coiso} with the above
conventions. Finally, the construction of filtered Floer--Novikov
homology for open manifolds, utilized in Section \ref{sec:general}, is
briefly reviewed in Section~\ref{sec:FN}.

\subsection{Proof of Theorem \ref{thm:main}: a particular case}
Now we are in a position to prove Theorem \ref{thm:main} in the
particular case where $P$ is geometrically bounded and symplectically
aspherical. First observe that $H$ has a non-trivial contractible
one-periodic orbit $\gamma$ with
\begin{equation}
\labell{eq:index1}
1-\dim M = n_0-\dim P/2 \leq \Delta(\gamma)\leq n_0+\dim P/2=1+\codim M.
\end{equation}
Indeed, let $\tH\colon S^1\times P\to \R$ be a compactly supported,
$C^1$-close to $H$, non-degenerate perturbation of $H$. By Proposition
\ref{prop:floer}, $\tH$ has a non-degenerate contractible orbit
$\tgamma$ with action in the interval $(a,\,b)$ and Conley--Zehnder
index $n_0$.  By \eqref{eq:CZ-SZ},
$$
n_0-\dim P/2 \leq \Delta(\tgamma)\leq n_0+\dim P/2.
$$
Passing to the limit as $\tH\to H$ and setting $\gamma=\lim \tgamma$,
we conclude that the same is true for $\Delta(\gamma)$ by continuity
of $\Delta$. The orbit $\gamma$ is non-trivial since the trivial
orbits of $H$ have action either zero or $\max H$ while $A_H(\gamma)>
a >\max H$. As a consequence, $\gamma$ lies on a level of $H$ with
$r^2-\eps<K<r^2+\eps$.

Since $H$ is a function of $K$, we may also view $\gamma$, keeping the
same notation for the orbit, as a $T$-periodic orbit of $K$. Note that
$H$ is a decreasing function of $K$, but otherwise the requirements of
Lemma \ref{lemma:change} are met.  Hence, $\Delta(\gamma,
K)=-\Delta(\gamma)$, where $\Delta(\gamma)=\Delta(\gamma,H)$. Thus
\eqref{eq:index1} turns into
$$
1-\dim M \leq -\Delta(\gamma,K)\leq 1+\codim M.
$$
On the other hand, up to a sign, inequality \eqref{eq:growth} of
Proposition \ref{eq:growth} still holds for $\gamma$ with constants
$a>0$ and $c$ independent of $H$ and $r$ and $\eps>0$:
$$
-\Delta(\gamma,K)\geq a\cdot T-c.
$$
(The negative sign is a result of the convention change.) Hence, we
have an \emph{a priori} bound on $T$:
$$
T\leq T_0=(1+c+\codim M)/a.
$$
Passing to the limit as $\eps\to 0$, we see that the $T$-periodic
orbits $\gamma$ of $K$ converge, by the Arzela-Ascoli theorem, to a
periodic orbit of $K$ on the level $K=r^2$ with period bounded from
above by $T_0$. This completes the proof of Theorem \ref{thm:main} in
the particular case.

\begin{Remark}
In the proof above, the arguments from \cite{CGK,Gu,Ke} could also be
used in place of the result from \cite{GG}. The only reason for
utilizing that particular result is that its proof affords an easy,
essentially word-for-word, extension to the general case.
\end{Remark}

\subsection{Proof of Proposition \ref{prop:inf-many}}
\labell{sec:pf-inf-many}

In the setting of the proposition, fix a trivialization of the normal
bundle $(TM)^\omega$. Then, every energy level $K=r^2$ also inherits a
trivialization via its identification with the unit sphere bundle in
$(TM)^\omega$. When $r>0$ is small, the field of directions
$\ker(\omega|_{K=r^2})$ is transverse to the horizontal
distribution. Hence, fixing a horizontal section $\Gamma$ of the
$S^1$-bundle $\{K=r^2\}\to M$, we obtain the Poincar\'e return map
$\varphi\colon \Gamma\to\Gamma$.  Clearly, periodic points of
$\varphi$ are in one-to-one correspondence with periodic orbits of the
Hamiltonian flow on $K=r^2$.  The restriction $\omega|_\Gamma$ is a
symplectic form preserved by $\varphi$. Furthermore, as is easy to
see, $\Gamma$ is symplectomorphic to $M$. Thus, we can view $\varphi$
as a symplectomorphism $M\to M$. It is not hard to show that $\varphi$
is in fact a Hamiltonian diffeomorphism; see \cite{Gi:FA} and also
\cite{Ar2}.  The proposition now follows from the Conley conjecture
proved in \cite{Gi:conley}.

\section{Filtered Floer--Novikov homology for open manifolds}
\labell{sec:FN}

In this section, we describe a version of Floer (or Floer--Novikov)
homology which is suitable for extending Proposition \ref{prop:floer}
beyond the class of symplectically aspherical, geometrically bounded
manifolds.  We will focus on the case where $P$ is open, but not
necessarily geometrically bounded, for this is the setting most
relevant to the proof of Theorem \ref{thm:main}. Furthermore, we also
assume throughout the construction that $P$ is spherically rational,
i.e., $\left<\omega,\pi_2(P)\right>=\lambda_0\Z$ for some $\lambda_0>
0$ or $\omega|_{\pi_2(P)}=0$. In the latter case, it is convenient to set
$\lambda_0=\infty$ and $\lambda_0\Z=\{0\}$.

\subsection{Definitions}
Fix an open set $W\subset P$ with compact closure. Let $H\colon
S^1\times P\to \R$ be a one-periodic Hamiltonian on $P$, supported in
$W$ (or rather in $S^1\times W$). Two cappings $v_0$ and $v_1$ of the
same one-periodic orbit $\gamma$ of $H$ are said to be equivalent if
$\left<\omega,w\right>=0=\left<c_1(TP),w\right>$, where $w\in
\pi_2(P)$ is the sphere obtained by attaching $v_1$ to $v_0$ along
$\gamma$.  (For instance, when $P$ is symplectically aspherical any
two cappings are equivalent.) The value of the action functional
$A_H(\gamma,v)$ and the Conley--Zehnder index $\MUCZ(\gamma,v)$ and
the Salamon--Zehnder invariant $\Delta_v(\gamma)$ are entirely
determined by the equivalence class of $v$ and from now on we do not
distinguish equivalent cappings.

Assume that all one-periodic orbits of $H$ with action in
$(0,\,\lambda_0)$ are non-degenerate and that there are only finitely
many such orbits. This is a $C^\infty$-generic condition in $H$, cf.\
\cite{FHS,HS}. (However, since $P$ is open and $H$ is compactly
supported, the flow necessarily has trivial periodic orbits. Such
orbits are degenerate and have action in $\lambda_0\Z$.) For
$0<a<b<\lambda_0$ denote by $\CF^{(a,\,b)}_k(H)$ the vector space
freely generated over $\Z_2$ by (capped) orbits $x=(\gamma, v)$ with
$\MUCZ(x)=k$ and $a<A_H(x)<b$. Note that each vector space
$\CF^{(a,\,b)}_k(H)$ has finite dimension, for changing the
equivalence class of a capping by attaching a sphere $w$ to it shifts
the action and the index by $\left<\omega,w\right>\in \lambda_0\Z$
and, respectively, $2\left<c_1(TP),w\right>$.

Fix an almost complex structure $J$ (compatible with $\omega$) on $P$,
which we allow to be time-dependent within $W$. We define the Floer
differential $\p\colon \CF^{(a,\,b)}_k(H)\to \CF^{(a,\,b)}_{k-1}(H)$
by the standard formula:
\begin{equation}
\labell{eq:d-floer}
\p x=\sum \#[\widehat{\CM}(x,y)]\cdot y,
\end{equation}
where $x$ is a capped orbit $(\gamma^-,v^-)$, the sum is taken over
all $y=(\gamma^+,v^+)$ with index $k-1$ and action in $(a,\,b)$, and
$\#[\widehat{\CM}(x,y)]$ is the number ($\!\!\!\!\mod 2$) of points in
the moduli space $\widehat{\CM}(x,y)$ of Floer anti-gradient
connecting trajectories from $x$ to $y$. Let us recall the definition
of this moduli space.

Let $(s,t)$ be the coordinates on $\R\times S^1$. Denote by $\CM(x,y)$
the space formed by solutions $u\colon \R\times S^1\to P$ of Floer's
equation
\begin{equation}
\label{eq:floer}
\frac{\p u}{\p s}+ J_t(u) \frac{\p u}{\p t}=-\nabla H_t(u)
\end{equation}
which are asymptotic to $\gamma^\pm$ at $\pm\infty$, i.e., $u(s,t)\to
\gamma^\pm(t)$ point-wise as $s\to \pm\infty$, and such that the
capping $v^-$ is equivalent to the one obtained by attaching $u$ to
$v^+$ along $\gamma^+$. The space $\CM(x,y)$ carries an $\R$-action
given by shifts of $s$.  We set $\widehat{\CM}(x,y)=\CM(x,y)/\R$.  For
a generic Hamiltonian $H$ supported in $W$, the space $\CM(x,y)$,
equipped with the topology of uniform $C^\infty$-convergence on
compact sets, is a smooth manifold of dimension $\MUCZ(x)-\MUCZ(y)$;
\cite{FHS,HS}.  Furthermore, the $\R$-action on $\CM(x,y)$ is
non-trivial unless $\CM(x,y)$ is comprised entirely of one solution
$u(s,t)$ independent of $s$, and thus $\gamma^+=u=\gamma^-$ and $\dim
\CM(x,y)=0$. It follows that $\widehat{\CM}(x,y)$ is a discreet set
when $\MUCZ(x)=\MUCZ(y)+1$.

As is well known, one cannot expect to have $\p^2=0$ unless
$P$ satisfies some additional topological and geometrical
requirements. Moreover, the moduli spaces $\widehat{\CM}(x,y)$ in
\eqref{eq:d-floer} need not in general be finite and once one of these
sets is infinite $\p$ is not even defined.  The next lemma shows,
however, that these problems do not arise if the action interval
$(a,\,b)$ is sufficiently short.

\begin{Lemma}
\labell{lemma:floer}

There exists a constant $h=h(W,J)>0$, depending only on $W$ and $J$
but not on $H$, such that once $b-a<h$, the zero dimensional moduli
spaces $\widehat{\CM}(x,y)$ in \eqref{eq:d-floer} are finite and
$\p^2=0$.
\end{Lemma}

\begin{Remark}
As will be clear from the proof of Lemma \ref{lemma:floer}, the
constant $h>0$ can be chosen to be the same for all almost complex
structures that are sufficiently $C^1$-close to $J$. This is essential
to ensure generic regularity for Floer continuation maps.
\end{Remark}

The lemma is nearly obvious. Two phenomena can interfere with
compactness of $\widehat{\CM}(x,y)$ or cause $\p^2$ not to be zero:
bubbling-off and the existence of bounded energy sequences of Floer
connecting trajectories going to infinity in $P$.  However, since $H$
is supported in $W$, every Floer connecting trajectory is a
holomorphic curve outside $W$.  Leaving a compact neighborhood
$\bar{V}$ of $W$ requires such a curve to have energy exceeding some
$h_1(V,J)>0$. Furthermore, since $P$ is spherically rational,
bubbling-off requires energy $\lambda_0$ or greater. (Alternatively,
since Floer trajectories with energy smaller that $h_1$ are confined
to a compact set $\bar{V}$, one can invoke Gromov's compactness
theorem rather than rationality of $P$.)  Summarizing, we conclude
that neither of these phenomena can occur when
$b-a<h=\min\{h_1,\lambda_0\}$. For the sake of completeness, we
provide a more detailed argument.

\begin{proof}
Fix an open set $V\supset \bar{W}$ with compact closure. 
Without loss of generality we may assume that $\bar{V}$ and $\bar{W}$
are smooth connected manifolds with boundary. Then $Y=\bar{V}\ssminus
W$ is a smooth compact domain whose boundary has two components: $\p
\bar{W}$ and $\p \bar{V}$. There exists a constant $h_1=h_1(W,V,J)>0$
such that for every holomorphic curve $v$ in $Y$ whose boundary is
contained in $\p Y$ and meets both of the components of $\p Y$ we have
$$
E(v):=\int_v \omega> h_1.
$$

This fact is an immediate consequence of a result of Sikorav,
\cite[p.\ 179]{AL}.  Namely, consider a holomorphic curve through
$z\in P$ with boundary on the $R$-sphere centered at $z$. Then
according to this result, there exist constants $C$ and $R_0>0$
(depending only on $z$) such that the area of the holomorphic curve is
greater than $CR^2$ whenever $0<R<R_0$. Moreover, it is clear that $C$
and $R_0$ can be taken independent of $z$ as long as $z$ varies within
a fixed compact set. Let now $S$ be a closed hypersurface in $Y$
separating the two boundary components of $\p Y$. Then, any
holomorphic curve $v$ as above passes through a point $z\in S$, and
thus through a ball of radius $R>0$, where $R$ depends only on $S$,
centered at $z$ and contained in $Y$. Taking $R>0$ sufficiently small
and applying Sikorav's result, we conclude that $E(v)> CR^2=:h_1$.

Assume now that $b-a<h:=\min\{h_1,\lambda_0\}$. Then a Floer
trajectory $u$ connecting $x$ to $y$ with $0<a<A_H(y)\leq
A_H(x)<b<\lambda_0$ is necessarily contained in $V$.  Indeed, denote
by $v$ a part of $u$ lying in $Y=\bar{V}\ssminus W$. Clearly, $v$ is a
holomorphic curve (with boundary on $\p Y$) since $H$ is supported in
$W$, and
\begin{equation}
\labell{eq:energy}
E(v)\leq E(u):=
\int_{-\infty}^\infty \left\|\frac{\p u}{\p s}\right\|_{L^2(S^1)}^2\,ds
=A_H(x)-A_H(y)<h_1.
\end{equation}
Furthermore, the boundary of $v$ meets $\p \bar{W}$, for both $x$ and
$y$ are in $W$, and is entirely contained in $\p \bar{W}$ since
$E(v)<h_1$. Hence, $u$ takes values in $V$. Thus we have shown that
all connecting trajectories $u$ belong to a compact set $\bar{V}$.

Since we also have $b-a<\lambda_0$, bubbling-off is precluded by an
energy estimate similar to \eqref{eq:energy}. (Just a slightly more
elaborate argument utilizing Gromov's compactness theorem shows that
bubbling-off cannot occur whenever $b-a<h_2$ for some $h_2>0$, even if
$P$ is not spherically rational.) As a consequence, we infer by what
have now become standard arguments (see, e.g., \cite{HS,Sa}) that
$\widehat{\CM}(x,y)$ is compact when $\MUCZ(x)=\MUCZ(y)+1$, and hence
finite, and $\p^2=0$.
\end{proof}

From now on, when working with the Floer homology
$\HF^{(a,\,b)}_*(H)$, we will always assume that $0<a<b<\lambda_0$ and
$b-a<h$ and that $a$ and $b$ are outside the action spectrum $\CS(H)$
of $H$. (Note that, since $P$ is spherically rational and $H$ is
compactly supported, $\CS(H)$ is closed and has zero measure; see,
e.g., \cite{HZ,schwarz}.)

\subsection{Properties}
The Floer homology spaces defined above have the standard properties
of filtered Floer homology of compactly supported Hamiltonians on
geometrically bounded manifolds; see \cite{CGK,GG,Gi:coiso,Gu}. Here
we recall only three of these properties that are explicitly used in
the proof of the main theorem:

\begin{itemize}

\item For any three points $a<b<c$, where $0<a<c<\lambda_0$ and
$c-a<h$, we have the long exact sequence
$$
\ldots\to \HF^{(a,\,b)}_*(H)\to\HF^{(a,\,c)}_*(H)\to
\HF^{(b,\,c)}_*(H)\to\ldots\, .
$$
\item A monotone decreasing homotopy from a Hamiltonian $H^+$ to a
Hamiltonian $H^-$ gives rise to homomorphisms
$$
\Psi_{H^+H^-}\colon \HF^{(a,\,b)}_*(H^+)\to \HF^{(a,\,b)}_*(H^-),
$$
which are independent of the homotopy and commute with the long exact
sequence homomorphisms.

\item Let $H^s$ be a family of Hamiltonians continuously parametrized
by $s\in [0,\,1]$ and let $a(s)<b(s)$ be two continuous functions of
$s$ such that $a(s)<b(s)$ and all intervals $(a(s),\,b(s))$ satisfy
the above requirements. Assume also that $a(s)$ and $b(s)$ are outside
$\CS(H^s)$.  Then the groups $\HF^{(a(s),\,b(s))}_*(H^s)$ are
isomorphic for all $s$. As a consequence, by continuity, we have
$\HF^{(a,\,b)}_*(H)$ defined even when the orbits of $H$ are
degenerate.
\end{itemize}

The proofs of these properties and the constructions involved are
identical to those for geometrically bounded symplectically aspherical
manifolds (see \cite{CGK,Gi:coiso,GG,Ke}) which in turn follow closely
the proofs for closed or convex manifolds (see, e.g.,
\cite{BPS,FH,FS,Gi:conley, McSa,Sa,schwarz,Vi:fun} and references
therein).

\begin{Remark}
\labell{rmk:module}

It is worth pointing that, in contrast with the total Floer--Novikov
homology on a closed manifold, the filtered homology spaces
$\HF^{(a,\,b)}_*(H)$ are not modules over the Novikov ring $\Lambda$
of $P$. (See, e.g., \cite{FHS,Sa,McSa} for the definition of
$\Lambda$.) However, a part of the $\Lambda$-module structure is
retained by the filtered homology. Namely, note first that the
requirement that $0<a<b<\lambda_0$ can be replaced by a less
restrictive condition that $(a,\,b)$ contains no points of
$\lambda_0\Z$. Then attaching a sphere $w\in\pi_2(P)$ simultaneously
to all cappings gives rise to an isomorphism $\CF^{(a,\,b)}_*(H)\to
\CF^{(a+\alpha,\,b+\alpha)}_{*+\mu}(H)$ of Floer complexes, and hence
of Floer homology, where $\alpha=\left<\omega,w\right>$ and
$\mu=2\left<c_1(TP),w\right>$.  This is an analogue of the action of
the generators of $\Lambda$ in the total Floer complex (or homology)
of $H$.
\end{Remark}

\begin{Remark}
\labell{rmk:other-cases}

If $P$ is closed or geometrically bounded, some of the restrictions
made in the construction of the filtered Floer homology can be
relaxed. Namely, when $P$ is closed and $W=P$, no homological
conditions on $P$ or restrictions on the action interval $(a,\,b)$ are
needed, provided that $(a,\,b)$ is sufficiently short. (This readily
follows from Gromov's compactness theorem, which guarantees that no
bubbling-off can occur on Floer connecting trajectories with small
energy.) When $P$ is open and geometrically bounded, there is no need
to fix an open set $W$ with compact closure and the constant $h$ is
independent of the support of $H$. (This is a consequence of Sikorav's
version of the Gromov compactness theorem; see \cite{AL}.) However,
the assumptions that $P$ is spherically rational and that
$0<a<b<\lambda_0$, or at least that $(a,\,b)$ contains no points of
$\lambda_0\Z$, appear to be essential. In fact, it is not clear how to
define the filtered Floer homology of a compactly supported
Hamiltonian on a geometrically bounded open manifold without requiring
$P$ to be spherically rational.
\end{Remark}

\section{Floer homological counterpart in the general case}
\labell{sec:general}

Throughout this section, we assume that $P$ satisfies the following
two conditions:

\begin{itemize}
\item[(a)] $P$ is spherically rational and
\item[(b)] for any Hamiltonian on $P$, changing (the equivalence class
of) a capping of a
periodic orbit necessarily alters its action value, i.e.,
$\ker[\omega]|_{\pi_2(P)}\subset \ker c_1(TP)|_{\pi_2(P)}$.
\end{itemize}
Note that in the setting of Theorem \ref{thm:main} one can ensure that
these requirements are met, as a consequence of either (i) or (ii), by
replacing $P$ by a small neighborhood of $M$.

\subsection{Generalization of Proposition \ref{prop:floer}}
\labell{sec:Floer-gen}

Fix a symplectic tubular neighborhood $W$ of $M$ and an almost complex
structure $J$ on $P$.  Let $K\colon P\to \R$ be an autonomous
Hamiltonian attaining its Morse-Bott non-degenerate minimum $K=0$
along a closed symplectic submanifold $M\subset P$ and let $H$ be
defined exactly as in Section \ref{sec:GG1} with $\{K\leq
r^2+\eps_0\}\subset W$. By (a), in the notation of Section
\ref{sec:FN}, the filtered Floer homology $\HF^{(a,\,b)}_*(H)$ is
defined whenever $0<a<b<\lambda_0$ and $b-a<h=h(W,J)$. Then, the
following analogue of Proposition \ref{prop:floer} holds:

\begin{Proposition}
\labell{prop:floer2}

There exists a function $C(r)>0$ of $r>0$ such that $C(r)\to 0$ as
$r\to 0$ and, once $r>0$ is sufficiently small, we have
$$
\HF^{(a,\,b)}_{n_0}(H)\neq 0
$$
for any $H$ as above with $\max H=C(r)$ and for some interval
$(a,\,b)$ with $C(r)< a<b <2C(r)$. (Here, as in Proposition
\ref{prop:floer}, $n_0=1+(\codim M-\dim M)/2$.)
\end{Proposition}

\begin{Remark}
Since the proposition concerns only a small neighborhood of $M$,
requirement (a) can be replaced by the condition that $M$ is
spherically rational and, in (b), $[\omega]|_{\pi_2(P)}$ can be
replaced $[\omega]|_{\pi_2(M)}$.
\end{Remark}

\subsection{Proof of Proposition \ref{prop:floer2}}

To prove the proposition we will, as in \cite{GG}, construct functions
$F^\pm$ such that
$$
F^-\leq H\leq F^+
$$
and $\HF^{(a,\,b)}_{n_0}(F^\pm)\cong\Z_2$ and the monotone homotopy
map
$$
\Psi\colon \Z_2\cong \HF^{(a,\,b)}_{n_0}(F^+)\to \HF^{(a,\,b)}_{n_0}(F^-)
\cong\Z_2
$$
is an isomorphism if $r>0$ is sufficiently small.  Then
$\HF^{(a,\,b)}_{n_0}(H)\neq 0$, for $\Psi$ factors through
$\HF^{(a,\,b)}_{n_0}(H)$.  The argument closely follows, with some
simplifications, the proof of Proposition \ref{prop:floer} given in
\cite{GG} and we only briefly outline its key elements.

\subsubsection{Functions $F^\pm$ and the parameters $a$, $b$ and $C(r)$}

The almost complex structure $J$ and the symplectic form $\omega$ give
rise to a Hermitian metric on the normal bundle $(TM)^\omega$ to $M$.
Without loss of generality we may assume that $W$ is a tubular
neighborhood of $M$ in $P$; see Section \ref{sec:darboux}.  In
particular, $W$ is equipped with projection $W\to M$ whose fibers are
identified with fiber-wise balls in the normal bundle $(TM)^\omega$
and this identification preserves the symplectic structure. Denote by
$\rho \colon W\to \R$ the square of the Hermitian norm on
$(TM)^\omega$ divided by $4\pi$, i.e., $\rho(X)=\|X\|^2/(4\pi)$ when
$X$ is viewed as a point in $(TM)^\omega$. It is easy to see that all
levels of $\rho$ are comprised of one-periodic orbits of its
Hamiltonian flow.  These orbits are the Hopf circles lying in the
fibers and bounding symplectic area $4\pi^2\rho$; see, e.g.,
\cite{CGK,GG}.

The normal-direction Hessian $d^2_M K$ along $M$ can also be viewed as
a fiber-wise quadratic function $d^2_M K\colon W\to \R$.  (Since $K$
is Morse--Bott non-degenerate, $d^2_M K$ is a fiber-wise metric on
$(TM)^\omega$. In general, this metric is not Hermitian.)  Recall also
that $0<\eps<\eps_0=r^2/10$; see Section
\ref{sec:GG1}. Thus, when $r>0$ is small, the shell $r^2-\eps\leq
K\leq r^2+\eps$ is contained in the shell
$$
Z=\{r^2-2\eps_0\leq d^2_MK\leq r^2+2\eps_0\}.
$$
Set
$$
\rho_3^-=\min_Z \rho, \quad \rho_2^-=2\rho_3^-/3\quad\text{and}\quad
\rho_1^-=\rho_3^-/3,
$$
so that the points $\rho^-_1$ and $\rho^-_2$ divide the interval
$[0,\,\rho^-_3]$ into three equal parts. Furthermore, let, as in Fig.\ 1,
$$
\rho_1^+=\max_Z\rho, \quad \rho_2^+=\rho_1^+ + \rho_1^-
\quad\text{and}\quad \rho_3^+=\rho_1^+ + 2\rho_1^-,
$$
and
$$
C(r)=8\pi^2\rho_3^+.
$$
In other words, $C(r)/2$ is the symplectic area bounded by
one-periodic orbits of $\rho$ on the level $\rho=\rho_3^+$.  It is
clear that $C(r)\to 0$ as $r\to 0$.  From now on, we assume that $\max
H=C(r)$.

\begin{figure}[htbp]
\begin{center}
\input{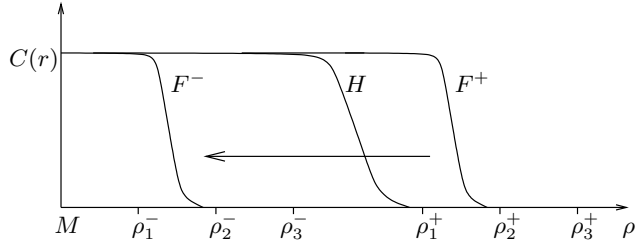} 
\caption{The functions $F^\pm$ and the homotopy.}
\label{figure:functions}
\end{center}
\end{figure}

The function $F^-\colon [0,\,\rho_3^-]\to\R$ is defined as follows
(see Fig.\ 1):
\begin{itemize}

\item $F^-(\rho)\equiv C(r)$ for $\rho\in [0,\rho_1^-]$;

\item $F^-(\rho)\equiv 0$ for $\rho\in [\rho_2^-,\rho_3^-]$;

\item $F^-(\rho)$ is a linear function on $[\rho_1^-,\rho_2^-]$
ranging from $C(r)$ to $0$ with irrational slope, except for $\rho$
close to $\rho_1^-$ and $\rho_2^-$ where $F^-$ is concave (with
decreasing $(F^-)'\leq 0$) and, respectively, convex (with increasing
$(F^-)'\leq 0$).

\end{itemize}
We extend the function $F^-\circ\rho$ from the domain
$\{\rho<\rho_3^-\}$ to $P$ by setting it to be identically zero
outside the domain and, abusing notation, refer to the resulting
Hamiltonian on $P$ as $F^-$.

It is essential that the ratio $C(r)/\rho_1^-$ is independent of $r$
(since $d^2_M K$ and $\rho$ are both fiber-wise quadratic), and hence
the slope of $F^-$ can also be taken independent of $r$.

Let also $F^+\colon [0,\,\rho_3^+]\to\R$ be defined by
\begin{itemize}

\item $F^+(\rho)\equiv C(r)$ for $\rho\in [0,\rho_1^+]$;

\item $F^+(\rho)=F^-(\rho-\rho_1^+ +\rho_1^-)$ for $\rho\in
[\rho_1^+,\rho_2^+]$;

\item $F^+(\rho)\equiv 0$ for $\rho\in [\rho_2^+,\rho_3^+]$.
\end{itemize}
In other words, $F^+$ is obtained from $F^-$ by shifting the graph of
$F^-$ to the left by $\rho_1^+ -\rho_1^-$ and extending it to the
remaining interval $[0,\,\rho_1^+ -\rho_1^-]$ as a function
identically equal to $C(r)$.  In particular, $F^+$ has the same slope
as $F^-$ and this slope is independent of $r$.  Finally, we extend
$F^+\circ\rho$ to $P$ in the same fashion as $F^-\circ\rho$ and again
keep the notation $F^+$ for the resulting Hamiltonian.

By the construction, $F^-\leq H\leq F^+$. Set
$$
a=C(r)+2\pi^2\rho_1^-\quad\text{and}\quad
b=C(r)+6\pi^2\rho_3^+.
$$
Then $a>C(r)$ and, since $C(r)=8\pi^2 \rho_3^+$, we have $b<2 C(r)$.
In what follows, we will assume that $r>0$ is sufficiently
small. Thus, in particular, $2 C(r)<\lambda_0$ and $b-a< C(r)<h(W,J)$,
and the Floer homology groups in question are defined.

\subsubsection{One-periodic orbits of $F^\pm$}

Trivial periodic orbits of $F^\pm$ are either the points where
$F^\pm=C(r)$ or the points where $F^\pm=0$.  Non-trivial one-periodic
orbits of the functions $F^\pm$ fill in entire energy levels of
$F^\pm$. We break down the corresponding energy values into two groups
$\rho=x^\pm_l$ and $\rho=y^\pm_l$ for each of the functions $F^\pm$;
see Fig.\ 2.

\begin{figure}[htbp]
\begin{center}
\input{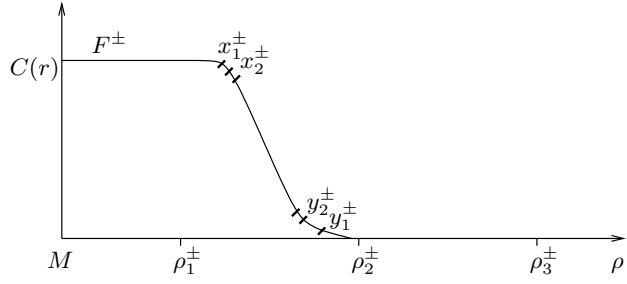} 
\caption{The energy levels $x_l^\pm$ and $y_l^\pm$.}
\label{figure:levels}
\end{center}
\end{figure}

The first group of levels is located in the region where $F^\pm$ is
convex. We label these levels by the corresponding values of $\rho$ in
the increasing order:
$$
x_1^\pm<x_2^\pm<\ldots<x_k^\pm,
$$
where all $x_j^\pm$ are close to and slightly greater than
$\rho_1^\pm$.

The levels from the second group are located in the region where
$F^\pm$ is concave. Again, we label these levels by the corresponding
values of $\rho$, but now in the decreasing order:
$$
y_k^\pm < \ldots < y_2^\pm < y_1^\pm,
$$
where all $y_l^\pm$ are close to and slightly smaller than
$\rho_2^\pm$.

Note that the number $k$ of levels in every group is completely
determined by the slope of the Hamiltonian. In particular, $k$ is the
same for all groups and is independent of $r$.

One-periodic orbits of $F^\pm$ on the levels $\rho=x_l^\pm$ and
$\rho=y_l^\pm$ are the fiber-wise Hopf circles traversed $l$-times.
We equip these orbits with cappings by discs contained in the fibers
and refer to this capping as fiber-wise. Denote by $A(x_l^\pm)$ and
$A(y_l^\pm)$ the resulting action values. (The action is independent
of the choice of an orbit on the level.) It is easy to see that
$$
A(x_l^\pm)=C(r)+4\pi^2 l x_l^\pm+\ldots=C(r)+4\pi^2 l \rho_1^\pm+\ldots
$$
and
$$
A(y_l^\pm)=4\pi^2 l y_l^\pm+\ldots=4\pi^2 l \rho_2^\pm+\ldots,
$$
where the dots denote an error that can be made arbitrarily small by
making $F^\pm$ close to a piece-wise linear function; see \cite{GG}.
Note that $A(x^\pm_1)$ is in $(a,\,b)$ while $A(y^\pm_1)$ and
$A(y^\pm_2)$ are outside this interval.

We require $r>0$ to be so small that all $A(x_l^\pm)$ and $A(y_l^\pm)$
are in the range from $0$ to $h(W,J)$.  This condition is indeed met
for small $r>0$, as $k$ is independent of $r$ and the largest of these
actions does not exceed $\max\{A(x^+_k),A(y^+_k)\}\leq C(r)+4\pi^2 k
\rho_2^\pm+\ldots$ .

\subsubsection{Floer homology of $F^\pm$ and the monotone homotopy map $\Psi$}

Let $(\alpha,\,\beta)$ be an interval in $(a,\,b)$ containing only one
point of $\CS(F^\pm)$. Denote by $\Sigma$ the unit sphere bundle in
$(TM)^\omega$. (Thus, the levels $\rho=x^\pm_l$ and $\rho=y^\pm_l$ are
diffeomorphic to $\Sigma$.)

\begin{Lemma}
\labell{lemma:Poz}
$\HF^{(\alpha,\,\beta)}_*(F^\pm)=H_{*-\kappa}(\Sigma;\Z_2)$, where the
shift of degrees $\kappa$ depends on the level.
\end{Lemma}

We defer the proof of the lemma to Section \ref{sec:Poz-pf} and
continue the proof of the proposition.

To determine the exact value of the shift $\kappa$, consider a
non-degenerate time-dependent perturbation $\tF^\pm$ of $F^\pm$ that
differs from $F^\pm$ only in small neighborhoods of the levels
$x^\pm_l$ and $y^\pm_l$. The perturbation $\tF^\pm$ can be explicitly
constructed (see \cite[Section 5.2.5]{GG}) so that every level
$x^\pm_l$ and $y^\pm_l$ splits into a number of non-degenerate orbits
contained in the fibers of $W$ and these orbits, equipped with
fiber-wise cappings, have Conley--Zehnder indices in the intervals
\begin{eqnarray*}
&&[(2l-1)q-m+1,\,(2l+1)q+m] \quad\text{for the level $\rho=x^\pm_l$ and}\\
&&[(2l-1)q-m,\,(2l+1)q+m-1] \quad\text{for the level $\rho=y^\pm_l$,}
\end{eqnarray*}
where $2m=\dim M$ and $2q=\codim M$. Note that the length of each of
these intervals is $2(m+q)-1=\dim \Sigma$. Hence, $\kappa$ is the left
end-point of the index interval, i.e.,
$$
\kappa=(2l-1)q-m+1\text{ for $\rho=x^\pm_l$ and }
\kappa=(2l-1)q-m\text{ for $\rho=y^\pm_l$.}
$$
In particular, $\kappa=n_0$ for $x^\pm_1$ and, by Lemma
\ref{lemma:Poz}, $\HF^{(\alpha_1,\,\beta_1)}_{n_0}(F^\pm)=\Z_2$ when
$(\alpha_1,\,\beta_1)$ is a small interval containing no other points
of $\CS(F^\pm)$ than $A(x^\pm_1)$.

Arguing as in \cite{GG}, we establish the equality
$\HF^{(a,\,b)}_{n_0}(F^\pm)=\Z_2$ by utilizing the Floer homology long
exact sequence; see Section \ref{sec:FN}.

First note that the only action values of $F^\pm$ and $\tF^\pm$ in
$(a,\,b)$ are those of the orbits from some of the levels $x^\pm_l$
and $y^\pm_l$, equipped with fiber-wise cappings. In other words, only
the fiber-wise cappings are relevant. This is a consequence of (b) and
the requirement that $r>0$ is so small that all $A(x^\pm_l)$ and
$A(y^\pm_l)$ are in $(0,\lambda_0)$. Consider now a family of
intervals with the left end-point sliding down from $\alpha_1$ to $a$
and the right end-point increasing from $\beta_1$ to $b$. The filtered
Floer homology of $F^\pm$ can change only when an end-point of the
interval moves through an action value. More specifically, the
homology in degree $n_0$ can be effected only when an end-point moves
through an action value $A(x^\pm_l)$ or $A(y^\pm_l)$ with index
interval containing $n_0-1$ or $n_0$ or $n_0+1$. The only action
values with index intervals containing $n_0-1$ or $n_0$ are
$A(y_1^\pm)$ and $A(y_2^\pm)$, which are, however, outside the
interval $(a,\,b)$ since $A(y^\pm_1)<A(y^\pm_2)<a$. The levels
$x^\pm_l$ with $l\geq 2$ and $y^\pm_l$ with $l\geq 3$ have index
intervals starting above $n_0+1$. Finally, $A(x_1^\pm)\in
(\alpha_1,\,\beta_1)\subset (a,\,b)$. Hence, as readily follows from
the long exact sequence,
$\HF^{(a,\,b)}_{n_0}(F^\pm)\cong\HF^{(\alpha_1,\,\beta_1)}_{n_0}(F^\pm)
\cong\Z_2$.

The fact that the monotone homotopy map $\Psi$ is an isomorphism is
established in a similar fashion. Consider a monotone homotopy $F^s$
from $F^+$ to $F^-$ indicated by the arrow in Fig.\ 1 and obtained by
sliding the graph of $F^+$ to the left until it matches the graph of
$F^-$.  By the long exact sequence, $\HF^{(a,\,b)}_{n_0}(F^s)$ can
change only when action values of $F^s$ with index interval containing
$n_0-1$ or $n_0$ or $n_0+1$ enter or leave the interval
$(a,\,b)$. This, however, never happens, as is clear from the
calculation of actions and index intervals for $F^\pm$; see \cite{GG}
for more details.

To complete the argument, it remains to prove Lemma \ref{lemma:Poz}.

\subsection{Local Floer homology and the proof of Lemma \ref{lemma:Poz}}
\labell{sec:Poz-pf}

When $P$ is geometrically bounded and symplectically aspherical, the
lemma is an immediate consequence of the results of Po\'zniak,
\cite{Poz}, and, in particular, of \cite[Corollary 3.5.4]{Poz}; see
also \cite{GG} and \cite{BPS}. Moreover, Corollary 3.5.4 from
\cite{Poz} can be extended to a broader class of manifolds to apply in
the setting of Lemma \ref{lemma:Poz}.  However, such an extension is
not entirely straightforward even though it is essentially a
consequence of Po\'zniak's calculation of local Lagrangian Floer
homology for clean intersections, and we prefer to give a simpler
\emph{ad hoc} proof.

\subsubsection{Local Floer homology}
In this section we briefly recall the construction of local Floer
homology for autonomous Hamiltonians, following \cite{F:hs}; see also,
e.g., \cite{BPS,Gi:conley,Poz}.

Consider an autonomous Hamiltonian $F$ on a spherically rational
symplectic manifold $P^{2n}$ or, more generally, on an open subset
of $P$ invariant under the flow of $F$. Let $\Sigma$ be a compact,
connected set of fixed points of $\varphi=\varphi_F^1$.  Then $\Sigma$
is automatically invariant under the flow of $F$ and hence comprised
entirely of one-periodic orbits of $F$. In what follows, we will also
assume that $\Sigma$ is isolated, i.e., every fixed point of $\varphi$
in some neighborhood of $\Sigma$ (an isolating neighborhood) is
necessarily a point of $\Sigma$.

Fix an isolating neighborhood $U$ of $\Sigma$ and let $\tF$ be a
 one-periodic in time perturbation of $F$ such that $F-\tF$ is
 $C^2$-small and supported in $U$, and all one-periodic orbits of $\tF$
 in $U$ are non-degenerate. Let $\CF_k(\tF,U)$ be the vector space
 over $\Z_2$ generated by capped one-periodic orbits of $\tF$ in $U$
 of index $k$. (Note that we do not require the cappings to be
 contained in $U$.) Furthermore, fix $\eps>0$ and a one-periodic
 almost complex structure $J$ on $P$, and define the Floer
 differential $\p_\eps\colon \CF_k(\tF,U)\to\CF_{k-1}(\tF,U)$ by the
 standard formula \eqref{eq:d-floer}, where now all Floer connecting
 trajectories are required to have energy smaller than $\eps$.

The standard argument shows that $\p_\eps^2=0$ whenever $\eps>0$ and
$\|F-\tF\|_{C^2}$ are small enough, and, moreover, the resulting local
Floer homology spaces $\HF_*(F,\Sigma)$ are independent of $\eps>0$,
$\tF$ and $J$; see \cite{F:hs} and also, e.g.,
\cite{BPS,Gi:conley,HS,McSa,Poz,Sa,SZ}. (To be more precise, here we
need to require that $\eps<\eps_0(U,J)$ and
$\|F-\tF\|_{C^2}<\delta_0(U,J,\eps)$. Then the Floer anti-gradient
trajectories connecting periodic orbits in $U$ are confined to $U$ due
to the energy estimates from, e.g., \cite{SZ,Sa}. Bubbling-off cannot
occur, for $\eps<\lambda_0$.  As a consequence, the standard
compactness and continuation arguments apply.)

Moreover, the complex $\CF_*(\tF,U)$ carries a natural action
filtration (see Section \ref{sec:FN}) and we denote the resulting
filtered local Floer homology by $\HF_*^{(a,\,b)}(F,\Sigma)$. This
homology spaces are well-defined only when the points $a$ and $b$ are
outside the action spectrum $\CS(F,\Sigma)$ of $F|_\Sigma$.  (By
definition, $\CS(F,\Sigma)$ is comprised of action values of
one-periodic orbits of $F$ in $\Sigma$ with all possible cappings.  It
is easy to see that $\CS(F,\Sigma)$ has zero measure and is closed and
nowhere dense since $P$ is spherically rational.)  The filtered Floer
homology is essentially ``localized'' at the points of $\CS(F,\Sigma)$
and hence here, in contrast with the global case, the filtration plays
a rather superficial role. We will use it only to distinguish
contributions from different cappings of the same orbit.

Filtered local Floer homology inherits, in an obvious way, most of the
properties of ordinary filtered Floer homology. We will need the
following standard invariance result (cf.\ \cite{BPS,Poz,Gi:conley,Vi:fun}):

\begin{itemize}
\item Let $F^s$ be a family of Hamiltonians continuously parametrized
by $s\in [0,\,1]$ and such that $\Sigma$ is an isolated set of
periodic orbits for all $F^s$. Let $a(s)<b(s)$ be two continuous
functions of $s$ such that $a(s)$ and $b(s)$ are outside
$\CS(F^s,\Sigma)$.  Then the groups
$\HF^{(a(s),\,b(s))}_*(F^s,\Sigma)$ are isomorphic for all $s$.
\end{itemize}

Let us now turn to some examples which are relevant to the proof.

\begin{Example}
Assume that $P$ is symplectically aspherical and hence $\CS(F,\Sigma)$
is comprised of one point $c\in \R$. Then
$\HF_*^{(a,\,b)}(F,\Sigma)=0$ when $c$ is outside $[a,\,b]$ and
$\HF_*^{(a,\,b)}(F,\Sigma)=\HF_*(F,\Sigma)$ if $a<c<b$. Assume
furthermore that $\Sigma$ is a Morse-Bott non-degenerate manifold of
fixed points of $\varphi$, i.e., $\ker(d\varphi_p-I)=T_p\Sigma$ for
all $p\in \Sigma$.  Then $\HF_*(F,\Sigma)=H_{*-\kappa}(\Sigma;\Z_2)$
as is proved in \cite{Poz}.
\end{Example}

\begin{Example}
\labell{ex:loc-fl}

In this example, we assume that $P^{2n}$ satisfies condition (b) in
addition to being spherically rational (condition (a)). Let $\Sigma$
be a Morse-Bott non-degenerate critical manifold of a smooth function
$F$ with, say, $F|_\Sigma=0$. Then $\Sigma$ is an isolated
set of fixed points of $\varphi$, and
\begin{equation}
\labell{eq:fh-coiso}
\HF^{(a,\,b)}_*(F,\Sigma)=H_{*-\kappa(F)}(\Sigma;\Z_2),
\end{equation}
whenever $\Sigma$ is a hypersurface (or, more generally, a coisotropic
submanifold) and $(a,\,b)$ is a short interval containing $0$. Here
$\kappa(F) = n - \mathrm{index}(\Sigma)$, where
$\mathrm{index}(\Sigma)$ is the index of $F$ at $\Sigma$. Note that
$\Sigma$ is also a Morse-Bott non-degenerate manifold of fixed points
of $\varphi$ since $\Sigma$ is coisotropic. Thus, if $P$ is
symplectically aspherical, the identification \eqref{eq:fh-coiso}
becomes a particular case of Po\'zniak's result mentioned in the
previous example.

To establish the general case of \eqref{eq:fh-coiso}, we argue as
follows. First note that $\HF^{(a,\,b)}_*(sF,\Sigma)$ does not change
as $s$ ranges from $1$ to some small value $s_0>0$ such that $s_0F$ is
$C^2$-small. (Here, again we use the assumption that $\Sigma$ is
coisotropic which guarantees that $\Sigma$ is an isolated fixed point
set for all $s\in (0,\,1]$. Furthermore, it is clear from (a) and (b)
that the end points $a$ and $b$ are not in $\CS(sF)$, provided that
$a$ and $b$ are sufficiently close to zero.)  Then,
$\HF^{(a,\,b)}_*(s_0 F,\Sigma)$ can be identified, up to a shift of
degree by $n$, with the local Morse homology of $F$ at $\Sigma$ by
arguing as in \cite{FHS,HS,SZ} and using again conditions (a) and (b);
cf.\ \cite{Gi:conley}.  (Alternatively, one can utilize the PSS
isomorphism, \cite{PSS}, not relying on (b).) Finally, it is a
standard fact that the local Morse homology in question is equal to
$H_{*}(\Sigma;\Z_2)$, up to a shift of degree by
$\mathrm{index}(\Sigma)$.
\end{Example}

Under suitable additional hypotheses, the filtered local Floer
homology of $F$ is equal to the filtered global Floer homology.
Namely, let $F$ and $P$ and $(a,\,b)$ be as in Section \ref{sec:FN}.
Assume that $(a,\,b)$ contains only one point of $\CS(F)$ and that
this point also belongs to $\CS(F,\Sigma)$.  Then, as is easy to see,
\begin{equation}
\labell{eq:loc-glob}
\HF_*^{(a,\,b)}(F,\Sigma)=\HF_*^{(a,\,b)}(F).
\end{equation}

\begin{Remark}
The construction of local Floer homology outlined above goes through
with obvious modifications even when $F$ is not autonomous; see
\cite{F:hs} and, e.g., \cite{BPS,Poz}. The only reason that $F$ is
assumed here to be independent of time is that this assumption makes
the construction much more explicit, simplifies the wording, and is
sufficient for the proof of Lemma \ref{lemma:Poz}.
\end{Remark}

\subsubsection{Proof of Lemma \ref{lemma:Poz}}
Let $F$ be one of the two functions $F^\pm$ and let
$\Sigma=\{\rho=z\}$, where $z=x^\pm_l$ or $z=y^\pm_l$, be one of the
levels in question. Consider the Hamiltonians
$F_0(\rho)=F(z)+F'(z)(\rho-z)$ and $f=F-F_0$ defined on a neighborhood
of $\Sigma$. Then the hypersurface $\Sigma$ is a Morse-Bott
non-degenerate critical manifold of $f$ and $f|_\Sigma=0$.  By Example
\ref{ex:loc-fl}, $\HF^{(\alpha',\,\beta')}_*(f,\Sigma)=H_{*-\kappa(f)}
(\Sigma;\Z_2)$, when $(\alpha',\,\beta')$ is a short interval
containing $0$.

Denote by $c$ the action of $F$ on $\Sigma$ with respect to the
fiber-wise capping, i.e., $c=A(z)$, and set
$\alpha=\alpha'+c$ and $\beta=\beta'+c$. We will show that
\begin{equation}
\labell{eq:id1}
\HF^{(\alpha,\,\beta)}_*(F,\Sigma)=
\HF^{(\alpha',\,\beta')}_{*-\kappa'}(f,\Sigma)
\end{equation}
for some shift of degrees $\kappa'$. This will prove the lemma, since
then, due to \eqref{eq:fh-coiso} and \eqref{eq:loc-glob},
$$
\HF^{(\alpha,\,\beta)}_*(F) = \HF^{(\alpha,\,\beta)}_*(F,\Sigma)=
H_{*-\kappa}(\Sigma;\Z_2),
$$
where $\kappa=\kappa'+\kappa(f)$. (Note that $c$ and $\kappa'$ can be
interpreted as the action and, respectively, the Maslov index of the
loop $\psi^t$, cf.\ \cite[Sections 2.3 and 3.2]{Gi:conley}.)

To establish \eqref{eq:id1}, consider a perturbation $\tf$ of $f$ in a
small neighborhood $U$ of $\Sigma$ as in the construction of local
Floer homology.  Without loss of generality we may assume that $\tf$
is autonomous and all one-periodic orbits of $\tf$ in $U$ are critical
points of $\tf$ and that all such critical points are located on
$\Sigma$. Then these orbits enter $\CF_*^{(\alpha',\,\beta')}(\tf,U)$
equipped with trivial cappings. For any other capping would
necessarily, by (b), move the action outside the range
$(\alpha',\,\beta')$. Let $J=J_t$ be a (time-dependent) almost complex
structure. A Floer anti-gradient trajectory $u$ connecting two
one-periodic orbits of $\tf$ in $U$ is a sphere.  (We are assuming
that $\tf$ is so close to $f$ that $u$ is contained in $U$.)  Since,
by the definition of local Floer homology, the energy of $u$ is small
and the values of $\tf$ at its critical points are close to zero, the
symplectic area of $u$ is small. Therefore, by (a) and (b),
\begin{equation}
\labell{eq:zero}
\left<\omega,u\right>=0=\left<c_1(TP),u\right>.
\end{equation}

The Hamiltonian flow of $F_0$ is a one-periodic loop
$\psi^t=\varphi^t_{F_0}$ of fiber-wise rotations and the
Hamiltonian $\tF=F_0+\tf\circ (\psi^t)^{-1}$ generating the flow
$\psi^t\circ\varphi^t_{\tf}$ is a small perturbation of $F$.  Denote
by $\HF_*^{(\alpha',\,\beta')}(\tf,U)$ and
$\HF_*^{(\alpha,\,\beta)}(\tF,U)$ the homology of the complexes of
$\CF_*^{(\alpha',\,\beta')}(\tf,U)$ and, respectively,
$\CF_*^{(\alpha,\,\beta)}(\tF,U)$. By the definition of local Floer
homology,
$$
\HF_*^{(\alpha',\,\beta')}(\tf,U)=\HF^{(\alpha',\,\beta')}_*(f,\Sigma)
\quad\text{and}\quad
\HF_*^{(\alpha,\,\beta)}(\tF,U)=\HF^{(\alpha,\,\beta)}_*(F,\Sigma),
$$
and \eqref{eq:id1} is equivalent to the isomorphism
\begin{equation}
\labell{eq:id2}
\HF_*^{(\alpha,\,\beta)}(\tF,U) \cong
\HF_{*-\kappa'}^{(\alpha',\,\beta')}(\tf,U).
\end{equation}
This isomorphism is induced by the composition with the loop $\psi^t$.

Indeed, the composition with $\psi^t$ gives rise to a one-to-one
correspondence between one-periodic orbits of $\tf$ contained in $U$
and those of $\tF$, and the latter are fiber-wise Hopf circles
(perhaps, multi-covered). Furthermore, as is well known,
$\psi(u)(s,t):=\psi^t(u(s,t))$ is a Floer anti-gradient trajectory for
$\tF$ in $U$ with respect to the almost complex structure
$\psi(J):=d\psi^t \circ J_t \circ (d\psi^t)^{-1}$ if and only if $u$
is a Floer anti-gradient trajectory for $\tf$ and $J$ contained in
$U$; see, e.g., \cite{Gi:conley,schwarz}. (Moreover, the regularity
requirements are satisfied for $(\tf, J)$ if and only if they are
satisfied for $(\tF,\psi(J))$.)

Let us equip the one-periodic orbits of $\tF$ contained in $U$ with
fiber-wise cappings. Then these capped orbits are the only orbits
entering $\CF_*^{(\alpha,\,\beta)}(\tF,U)$, as again follows from (a)
and (b).  Hence, to establish \eqref{eq:id2} and thus finish the proof
of the lemma, it is sufficient to show that the Floer connecting
trajectories $\psi(u)$ are compatible with such cappings.  In other
words, it remains to prove that whenever $\psi(u)$ is a Floer
anti-gradient trajectory from $\gamma_0$ and $\gamma_1$ and $v_0$ and
$v_1$ are cappings of $\gamma_0$ and $\gamma_1$ by fiber-wise Hopf
discs (perhaps, multi-covered), the capping $v_1$ is equivalent to the
capping $v_0\# \psi(u)$ obtained by attaching $v_0$ to $\psi(u)$. To
this end, consider the sphere $w$ obtained by attaching the suitably
oriented discs $v_0$ and $v_1$ to $\psi(u)$. The cappings $v_0\#
\psi(u)$ and $v_1$ are equivalent if and only if
$$
\left<\omega,w\right>=0=\left<c_1(TP),w\right>.
$$
Let $\pi\colon W\to M$ be the tubular neighborhood projection. Since
$v_0$ and $v_1$ lie in the fibers of $\pi$, the projections $\pi(v_0)$
and $\pi(v_1)$ are points. Hence, $\pi(w)$ is homotopic to
$l\cdot\pi(u)$, where $l\in\Z$, for $\psi^t$ is comprised of
fiber-wise rotations.  Then, by \eqref{eq:zero},
$$
\left<c_1(TP),w\right>=\left<c_1(TP|_M),\pi(w)\right>
=l\left<c_1(TP|_M),\pi(u)\right>=l\left<c_1(TP),u\right>=0.
$$
A similar calculation, showing that $\left<\omega,w\right>=0$,
completes the proof of \eqref{eq:id2} and of the lemma.

\section{Proof of the main theorem}
\labell{sec:gen-pf}

\subsection{Proof of Theorem \ref{thm:main}}
\labell{sec:pf-general}
When (i) holds, the proof of the main theorem in the general case is
identical word-for-word to the proof for geometrically bounded,
symplectically aspherical manifolds (Section \ref{sec:proof-sa}) with
Proposition \ref{prop:floer2} used in place of Proposition
\ref{prop:floer}.

To establish the theorem when (ii) holds, we use, in addition to the
Sturm comparison theorem for $K$, the action bounds from Proposition
\ref{prop:floer2} to control the effect of capping on the
Salamon--Zehnder invariant.

Fix small parameters $r>0$ and $\eps>0$ with $\eps<r^2/10$ and
consider a Hamiltonian $H$ as in Section \ref{sec:Floer-gen}. Recall
that (ii) implies that the hypotheses (a) and (b) of Section
\ref{sec:Floer-gen} are satisfied. Then, as in Section
\ref{sec:proof-sa}, it readily follows from Proposition
\ref{prop:floer2} that $H$ has a (non-trivial) one-periodic orbit
$\gamma$ with capping $v$ such that
\begin{equation}
\labell{eq:index2}
1-\dim M = n_0-\dim P/2 \leq \Delta_v(\gamma,H)\leq n_0+\dim P/2=1+\codim M
\end{equation}
and
$$
a\leq A_H(\gamma,v)\leq b.
$$
Recall that $0<C(r)<a<b<2C(r)$ and $0\leq H\leq C(r)$, where $C(r)\to
0$ as $r\to 0$. Hence, the symplectic area of $v$ is \emph{a priori}
bounded:
\begin{equation}
\labell{eq:area}
\left|\int_v\omega\right|\leq \const,
\end{equation}
where $\const$ is independent of $r$ and $\eps$ and the Hamiltonian
$H$. (Throughout the rest of the proof we adopt the notational
convention from Section \ref{sec:growth}: in all expressions $\const$
will stand for a constant which is independent of $r$, $\eps$, $H$,
and $(\gamma, v)$, provided that $r$ is sufficiently small. The value
of this constant is allowed to vary from one formula to another. A
similar convention is also applied to the constants $a>0$ and $b$ and
$c$.)

As in Section \ref{sec:proof-sa}, we may view $\gamma$ as a
$T$-periodic orbit of $K$ since $H$ is a function of $K$ and the orbit
$\gamma$ is non-trivial. Then, by \eqref{eq:length},
\begin{equation}
\labell{eq:length2}
l(\gamma)\leq \const \cdot r \cdot T.
\end{equation}

Fix a 2-form $\sigma$ representing $c_1(TP)$. By (ii),
$\sigma=\lambda\omega+d\alpha$, with $\lambda\neq 0$, for some 1-form
$\alpha$. Then
$$
\left|\int_v\sigma\right|\leq \left|\lambda\int_v\omega\right|
+\left|\int_\gamma\alpha\right|,
$$
and, from \eqref{eq:area} and \eqref{eq:length2}, we see that
\begin{equation}
\labell{eq:sigma}
\left|\int_v\sigma\right| \leq \const_1\cdot r\cdot T+\const_2.
\end{equation}

By Lemma \ref{lemma:change}, we have $\Delta_v(\gamma,
K)=-\Delta_v(\gamma,H)$, where the negative sign is a consequence of
the fact that $H$ is a decreasing function of $K$. Thus,
\eqref{eq:index2} turns into
$$
1-\dim M \leq -\Delta_v(\gamma,K)\leq 1+\codim M,
$$
and, by Proposition \ref{prop:growth'},
$$
-\Delta_v(\gamma,K)\geq a\cdot T-c -2\int_v\sigma,
$$
where $a>0$. (The negative sign is a result of the convention change
from Section \ref{sec:growth} to Section~\ref{sec:proof-sa}.) Therefore,
$$
T-\frac{2}{a}\int_v\sigma\leq \frac{c+1+\codim M}{a}.
$$
Combining this upper bound with \eqref{eq:sigma}, we conclude that if
$r>0$ is sufficiently small, $T\leq T_0$ for some $T_0$ that depends
only on $K$.

Finally, as in the proof of the particular case, passing to the limit
as $\eps\to 0$, we see that the $T$-periodic orbits $\gamma$ of $K$
converge, by the Arzela-Ascoli theorem, to a periodic orbit of $K$ on
the level $K=r^2$ with period bounded from above by $T_0$. This
completes the proof of Theorem \ref{thm:main}.

\begin{Remark}
As is readily seen from the proofs of Theorem \ref{thm:main} and
Proposition \ref{prop:floer2}, one can also estimate the action and
the symplectic area of the orbit $\gamma$ of $K$ on the level
$K=r^2$. Namely, let $v$ be the capping of $\gamma$ as in the proof of
Theorem \ref{thm:main}.  Then,
$$
-\const_1 \cdot r^2\leq \int_v\omega \leq -\const_2 \cdot r^2
\quad\text{and}\quad
|A_K(\gamma,v)|\leq \const \cdot r^2
,
$$
where all constants are positive and depend only on $K$.  This follows
from the facts that $\const_1 \cdot r^2\leq C(r) <\const_2 \cdot r^2$
with, perhaps, some other values of the constants and that the period
of $\gamma$ is bounded from above.

\end{Remark}

\subsection{Concluding remarks}
\labell{sec:role}

\subsubsection{Dense existence of periodic orbits}

Proposition \ref{prop:floer2} can be reformulated (with some loss of
information) as a dense existence theorem for periodic orbits of $K$:

\begin{Proposition}
\labell{prop:floer3}

Assume that $P$ satisfies conditions (a) and (b) of Section
\ref{sec:Floer-gen}. Then, for a dense set of small $r>0$, the level
$K=r^2$ carries a contractible in $P$ periodic orbit $\gamma$ with
capping $v$ such that
$$
-1-\codim M\leq \Delta_v(\gamma)\leq \dim M-1
\quad\text{and}\quad 0<\left|\int_v\omega\right|<\const\cdot r^2,
$$
where $\const$ depends only on $K$.
\end{Proposition}

Referring the reader back to Section \ref{sec:other-results} for a
discussion of other dense or almost existence results, here we only
point out that almost existence of contractible periodic orbits of $K$
without upper and lower bounds on $\Delta_v(\gamma)$ is proved in
\cite{Lu2} under no topological assumptions on $P$. Proposition
\ref{prop:floer3} is sufficient for the proof of Theorem
\ref{thm:main} and can be easily established as a consequence of
Proposition \ref{prop:floer2} by passing from periodic orbits of $H$
to those of $K$ as in Sections \ref{sec:proof-sa} and
\ref{sec:pf-general}.

\subsubsection{The role of hypotheses (i) and (ii)}
As is mentioned in Section \ref{sec:intro}, hypotheses (i) and
(ii) in Theorem \ref{thm:main} can possibly be relaxed. Indeed, the
Sturm theoretic counterpart of the proof (Proposition
\ref{prop:growth'}) requires no topological assumptions on $P$ or
$M$. The Floer homological part of the argument (Proposition
\ref{prop:floer2}) holds under hypotheses (a) and (b), less
restrictive than (i) or (ii), and can probably be extended to, at least,
all spherically rational manifolds by using, for instance, the
machinery of central Floer homology from \cite{Ke:central}; see also
\cite{Al}. Furthermore, in the form of Proposition \ref{prop:floer3},
it can perhaps be generalized to arbitrary symplectic manifolds by
utilizing the holomorphic curve techniques as in, e.g., \cite{Lu2}.
However, it is not clear to the authors how to combine these two
counterparts to obtain an upper bound on the period without using
conditions (i) or (ii) or some other condition relating $[\omega]$ and
$c_1(TP)$.

\subsubsection{Action control and ``contact homology'' approach}
We conclude this paper by discussing two approaches to proving Theorem
\ref{thm:main}, which are more natural than the one used here but
encounter a serious difficulty.

The key to the first approach lies in establishing an upper bound on
the period of an orbit of $H$ via its action. Then, the theorem would
follow directly from a version of Proposition \ref{prop:floer2}. This
method has been used, for instance, to prove the Weinstein conjecture
for hypersurfaces (of contact type) as a consequence of a calculation
of Floer or symplectic homology; see, e.g., \cite{FHW,HZ} and
\cite{Gi:alan} for further references. Here the condition that the
level in question has contact type is crucial for controlling the
period of an orbit via its action. This can be seen, for instance,
from the counterexamples to the Hamiltonian Seifert conjecture,
\cite{Gi:bayarea,Gi:barcelona,GG:ex,GG,Ke:ex}. In the setting of
Theorem \ref{thm:main}, the energy levels $S=\{K=r^2\}$ do not in
general have contact type (with very few exceptions), and the authors
are not aware of any way to relate the period and the action in this
case by merely using the fact that $S$ is fiber-wise convex.

The idea of the second approach is to make use of a version of the
contact homology $\HC_*(S)$ defined for the level $S$ and detecting
closed characteristics on $S$.  (Strictly speaking, no construction of
$\HC_*$ applicable to the levels in question is available at the
moment.)  Then, one would consider the continuation map
$\HC_*(\Sigma^+)\to \HC_*(S) \to \HC_*(\Sigma^-)$, where $\Sigma^+$ is
a level of $\rho$ enclosing $S$ and $\Sigma^-$ is a level of $\rho$
enclosed by $S$. As in the proof of Proposition \ref{prop:floer2}, one
can expect this map to be non-zero, which would then yield
$\HC_*(S)\neq 0$. This argument relies on the assumption that the
groups $\HC_*$ are sufficiently invariant under deformations of the
level. However, to the best of the authors' understanding, to
guarantee such invariance, sufficient control of period via action is
necessary as is indicated again by the counterexamples to the
Hamiltonian Seifert conjecture.  Hence, this approach encounters
essentially the same problem as the first one.


\begin{thebibliography}{LMc2}

\bibitem[Al]{Al}
P. Albers,
A note on local Floer homology, Preprint 2006, math.SG/0606600.


\bibitem[Ar1]{Ar}
V.I. Arnold,
Sturm theorems and symplectic geometry, \emph{Funktsional.\ Anal.\ i
Prilozhen.}, \textbf{19} (1985), 1--10.


\bibitem[Ar2]{Ar2}
V.I. Arnold,
On some problems in symplectic topology, in \emph{Topology and
Geometry -- Rochlin Seminar}, O.Ya.\ Viro (Editor), Lect.  Notes in
Math., vol.\ 1346, Springer, 1988.


\bibitem[AL]{AL}
M. Audin, J. Lafontaine (Eds),
\emph{Holomorphic Curves in Symplectic Geometry}, Progress in
Mathematics, \textbf{117}, Birkh\"auser Verlag, Basel, 1994.


\bibitem[BG]{BG}
J. Barge, E. Ghys,  
Cocycles d'Euler et de Maslov, \emph{Math.\ Ann.}  \textbf{294}
(1992), 235--265.


\bibitem[BPS]{BPS}
P. Biran, L. Polterovich, D. Salamon,
Propagation in Hamiltonian dynamics and relative symplectic homology,
\emph{Duke Math.\ J.}, \textbf{119} (2003), 65--118.


\bibitem[Bo]{Bo}
R. Bott,
On the iteration of closed geodesics and the Sturm intersection theory
\emph{Comm.\ Pure Appl.\ Math.}, \textbf{9} (1956), 171--206.


\bibitem[CGK]{CGK}
K. Cieliebak, V. Ginzburg, E. Kerman,
Symplectic homology and periodic orbits near symplectic submanifolds,
\emph{Comment.\ Math.\ Helv.}, \textbf{79} (2004), 554--581.


\bibitem[CZ1]{CZ1}
C.C. Conley, E. Zehnder,
Birkhoff--Lewis fixed point theorem and a conjecture of V.I. Arnold,
\emph{Invent.\ Math.}, \textbf{73} (1983), 33--49.


\bibitem[CZ2]{CZ2}
C.C. Conley, E. Zehnder,
Morse--type index theory for flows and periodic solutions of Hamilton
equations, \emph{Comm.\ Pure Appl.\ Math.}, \textbf{37} (1984),
207--253.


\bibitem[Co]{Co}
G. Contreras,
The Palais-Smale condition for contact type energy levels for convex
lagrangian systems, Preprint 2003, math.DS/0304238.


\bibitem[CIPP]{CIPP}
G. Contreras, R. Iturriaga, G. Paternain, M. Paternain,
The Palais-Smale condition and Ma\~n\'e's Critical Values.
\emph{Ann.\ Henri Poincar\'e}, \textbf{1} (2000), 655--684.


\bibitem[CMP]{CMP}
G. Contreras, L. Macarini, G. Paternain,
Periodic orbits for exact magnetic flows on surfaces,
\emph{IMRN}, 2005, no.\ 8, 362--387.


\bibitem[Du]{Du}
J.L. Dupont,
Bounds for characteristic numbers of flat bundles, in \emph{Proc.\
Sympos., Univ.\ Aarhus, Aarhus, 1978}, Lecture Notes in Math., 763,
pp. 109--119, Springer, Berlin, 1979.


\bibitem[Ed]{Ed}
H.M. Edwards,
A generalized Sturm theorem, \emph{Ann.\ of Math.}, \textbf{80}
(1964), 22--57.


\bibitem[EP1]{Po}
M. Entov, L. Polterovich,
Private communication.


\bibitem[EP2]{EP}
M. Entov, L. Polterovich,
Rigid subsets of symplectic manifolds, Preprint 2007,
arXiv:0704.0105[math.SG].


\bibitem[Fl1]{F:Morse}
A. Floer,
Morse theory for Lagrangian intersections.  \emph{J. Differential Geom.},
\textbf{28}  (1988),  513--547.


\bibitem[Fl2]{F:grad}
A. Floer,
The unregularized gradient flow of the symplectic action, \emph{Comm.\ Pure
Appl.\ Math.},  \textbf{41}  (1988),  775--813.


\bibitem[Fl3]{F:c-l}
A. Floer,
Cuplength estimates on Lagrangian intersections,  \emph{Comm.\ Pure Appl.\
Math.}, \textbf{42}  (1989), 335--356.


\bibitem[Fl4]{F:witten}
A. Floer,
Witten's complex and infinite-dimensional Morse theory,
\emph{J. Differential Geom.}, \textbf{30} (1989), 207--221.


\bibitem[Fl5]{F:hs}
A. Floer,
Symplectic fixed points and holomorphic spheres,  \emph{Comm.\ Math.\ Phys.},
\textbf{120}  (1989), 575--611.


\bibitem[FH]{FH}
A. Floer, H. Hofer,
Symplectic homology, I. Open sets in ${\mathbb C}^n$, \emph{Math.\
Z.}, \textbf{215} (1994), 37--88.


\bibitem[FHS]{FHS}
A. Floer, H. Hofer, D. Salamon,
Transversality in elliptic Morse theory for the symplectic action,
\emph{Duke Math.\ J.}, \textbf{80} (1995), 251--292.


\bibitem[FHW]{FHW}
A. Floer, H. Hofer, K. Wysocki,
Applications of symplectic homology, I. \emph{Math.\ Z.}, \textbf{217} (1994),
577--606.


\bibitem[FrHa]{FrHa}
J. Franks, M. Handel,
Periodic points of Hamiltonian surface diffeomorphisms, \emph{Geom.\ Topol.},
\textbf{7} (2003), 713--756.


\bibitem[FS]{FS}
U. Frauenfelder, F. Schlenk,
Hamiltonian dynamics on convex symplectic manifolds,
\emph{Israel J.  Math.}, \textbf{15} (2006).


\bibitem[Gi1]{Gi:FA}
V.L. Ginzburg,
New generalizations of Poincar\'{e}'s geometric theorem,
\emph{Functional Anal.\ Appl.}, \textbf{21} (2) (1987), 100--106.


\bibitem[Gi2]{Gi96}
V.L. Ginzburg,
On the existence and non-existence of closed trajectories for some
Hamiltonian flows, \emph{Math.\ Z.}, \textbf{223} (1996),  397--409.


\bibitem[Gi3]{Gi:newton}
V.L. Ginzburg,
On closed trajectories of a charge in a magnetic field. An
application of symplectic geometry, in \emph{Contact and Symplectic
Geometry}, C.B. Thomas (Ed.),  INI
Publications, Cambridge University Press, Cambridge, 1996, pp.\ 131--148.


\bibitem[Gi4]{Gi:bayarea} 
V.L. Ginzburg, 
Hamiltonian dynamical systems without periodic orbits, in
\emph{Northern California Symplectic Geometry Seminar}, 35--48, Amer.\
Math.\ Soc.\ Transl.\ Ser.\ 2, vol.\ 196, Amer.\ Math.\ Soc.,
Providence, RI, 1999.


\bibitem[Gi5]{Gi:barcelona}
V.L. Ginzburg,
The Hamiltonian Seifert conjecture: examples and open problems,
in \emph{Proceedings of the Third European Congress of Mathematics,
Barcelona, 2000}; Birkh\"auser, Progress in Mathematics, \textbf{202}
(2001), vol.\ II, pp.\ 547--555.


\bibitem[Gi6]{Gi:arnold}
V.L. Ginzburg,
Comments to some of Arnold's problems (1981-9 and related problems and
1994-13), in \emph{Arnold's problems}, Ed.: V.I. Arnold,
Springer--Verlag and Phasis, 2004; pp. 395--401, 557--558.  Also
available at

\verb!http:\\count.ucsc.edu\~ginzburg\publications.html!


\bibitem[Gi7]{Gi:alan}
V.L. Ginzburg,
The Weinstein conjecture and the theorems of nearby and almost
existence, in \emph{The Breadth of Symplectic and Poisson Geometry.
Festschrift in Honor of Alan Weinstein}; J.E. Marsden and T.S. Ratiu
(Eds.), Birkh\"auser, 2005, pp. 139--172.


\bibitem[Gi8]{Gi:coiso}
V.L. Ginzburg,
Coisotropic intersections, Preprint 2006, math.SG/0605186; to appear
in \emph{Duke Math.\ J.}


\bibitem[Gi9]{Gi:conley}
V.L. Ginzburg,
The Conley conjecture, Preprint 2006, math.SG/0610956.


\bibitem[GG1]{GG:ex}
V.L. Ginzburg, B.Z. G\"urel,
A $C^2$-smooth counterexample to the Hamiltonian Seifert
conjecture in $R^4$, \emph{Ann.\ of Math.}, \textbf{158} (2003), 953--976.


\bibitem[GG2]{GG}
V.L. Ginzburg, B.Z. G\"urel,
Relative Hofer--Zehnder capacity and periodic orbits in
twisted cotangent bundles, \emph{Duke Math.\ J.},
\textbf{123} (2004), 1--47.


\bibitem[GK1]{GK1}
V.L. Ginzburg, E. Kerman,
Periodic orbits in magnetic fields in dimensions greater than two,
in \emph{Geometry and Topology in Dynamics}, Ed.: M. Barge and
K. Kuperberg; Publ.\ of AMS, Cont.\ Math.\ Series, \textbf{246} (1999),
113--121.


\bibitem[GK2]{GK2}
V.L. Ginzburg, E. Kerman,
Periodic orbits of Hamiltonian flows near symplectic extrema,
\emph{Pacific J.\ Math.}, \textbf{206} (2002), 39--68.


\bibitem[G\"u]{Gu}
B.Z. G\"urel,
Totally non-coisotropic displacement and its applications to
Hamiltonian dynamics, Preprint 2007, math.SG/0702091.


\bibitem[Hi]{Hi}
N. Hingston,
Subharmonic solutions of Hamiltonian equations on tori, Preprint 2004;
to appear in \emph{Ann.\ of Math.}; available at

\verb!http://comet.lehman.cuny.edu/sormani/others/hingston.html!


\bibitem[HS]{HS}
H. Hofer, D. Salamon,
Floer homology and Novikov rings, in \emph{The Floer Memorial Volume},
484--524, Progr.\ Math., \textbf{133}, Birkh\"auser, Basel, 1995.


\bibitem[HZ1]{HZ:V}
H. Hofer, E. Zehnder,
Periodic solutions on hypersurfaces and a result by C. Viterbo,
\emph{Invent.\ Math.} \textbf{90} (1987), 1--9.


\bibitem[HZ2]{HZ:cap}
H. Hofer, E. Zehnder,
A new capacity for symplectic manifolds, in
\emph{Analysis, et cetera}, P. Rabinowitz and E. Zehnder (Eds.),
Academic Press, Boston, MA, 1990, pp.\ 405--427.


\bibitem[HZ3]{HZ}
H. Hofer, E. Zehnder,
\emph{Symplectic Invariants and Hamiltonian Dynamics}, Birk\"auser, 1994.


\bibitem[Ke1]{Ke1}
E. Kerman,
Periodic orbits of Hamiltonian flows near symplectic critical
submanifolds, \emph{IMRN}, 1999, no.\ 17, 954--969.


\bibitem[Ke2]{Ke:ex}
E. Kerman,
New smooth counterexamples to the Hamiltonian Seifert conjecture,
\emph{J.\ Symplectic Geometry}, \textbf{1} (2002), 253--267.


\bibitem[Ke3]{Ke}
E. Kerman,
Squeezing in Floer theory and refined Hofer-Zehnder capacities of sets
near symplectic submanifolds, \emph{Geom.\ Topol.}, \textbf{9} (2005)
1775--1834.


\bibitem[Ke4]{Ke:central}
E. Kerman,
Hofer's geometry and Floer theory under the quantum limit, Preprint 2007,
math.SG/0703064.


\bibitem[Ko]{Ko}
V.V. Kozlov,
Variational calculus in the large and classical mechanics,
\emph{Russian Math. Surveys}, \textbf{40} (2) (1985), 37--71.


\bibitem[Lu1]{Lu}
G.-C. Lu,
The Weinstein conjecture on some symplectic manifolds containing the
holomorphic spheres, \emph{Kyushu.\ J.\ Math.}, \textbf{52} (1998),
331--351.


\bibitem[Lu2]{Lu2}
G.-C. Lu,
Finiteness of Hofer-Zehnder symplectic capacity of neighborhoods of
symplectic submanifolds, Preprint 2005, math.SG/0510172.


\bibitem[Ma]{Ma}
L. Macarini,
Hofer--Zehnder capacity and Hamiltonian circle actions,
\emph{Commun.\ Contemp.\ Math.}, \textbf{6} (2004), no. 6, 913--945.


\bibitem[McSa]{McSa}
D. McDuff, D. Salamon, \emph{J-holomorphic Curves and Symplectic Topology},
Colloquium publications, vol.\ 52, AMS, Providence, RI, 2004.


\bibitem[Mo]{Mo:orbits}
J. Moser,
Periodic orbits near equilibrium and a theorem by Alan Weinstein,
\emph{Comm.\ Pure Appl.\ Math.}, \textbf{29} (1976), 727--747.


\bibitem[No]{No}
S.P. Novikov,
The Hamiltonian formalism and a many-valued analogue of Morse theory,
\emph{Russian Math.\ Surveys}, \textbf{37}(5) (1982), 1--56.


\bibitem[NT]{NT}
S.P. Novikov, I.A. Taimanov,
Periodic extremals of many-valued or not everywhere positive
functionals, \emph{Sov.\ Math.\ Dokl.}, \textbf{29}(1) (1984), 18--20.


\bibitem[PSS]{PSS}
S. Piunikhin, D. Salamon, M. Schwarz,
Symplectic Floer--Donaldson theory and quantum cohomology, in
\emph{Contact and Symplectic Geometry}, C.B. Thomas (Ed.), INI
Publications, Cambridge University Press, Cambridge, 1996, pp.\
151--170.


\bibitem[Poz]{Poz}
M. Po\'zniak,
Floer homology, Novikov rings and clean intersections, in
\emph{Northern California Symplectic Geometry Seminar}, 119--181,
Amer.\ Math.\ Soc.\ Transl.\ Ser.\ 2, \textbf{196}, Amer.\ Math.\
Soc., Providence, RI, 1999.


\bibitem[RS]{RS}
J. Robbin, D. Salamon,
The Maslov index for paths, \emph{Topology}, \textbf{32} (1993),
827--844.


\bibitem[Sa]{Sa}
D.A. Salamon,
Lectures on Floer homology, in \emph{Symplectic Geometry and
Topology}, Eds: Y. Eliashberg and L. Traynor, IAS/Park City
Mathematics series, \textbf{7}, 1999, pp.\ 143--230.


\bibitem[SZ]{SZ}
D. Salamon, E. Zehnder,
Morse theory for periodic solutions of Hamiltonian systems and the
Maslov index, \emph{Comm.\ Pure Appl.\ Math.}, \textbf{45} (1992),
1303--1360.


\bibitem[Schl]{Schl}
F. Schlenk,
Applications of Hofer's geometry to Hamiltonian dynamics,
\emph{Comment.\ Math.\ Helv.},  \textbf{81} (2006), 105--121.


\bibitem[Sc]{schwarz}
M. Schwarz,
On the action spectrum for closed symplectically aspherical manifolds,
\emph{Pacific J.\ Math.}, \textbf{193} (2000), 419--461.


\bibitem[St]{St}
M. Struwe,
Existence of periodic solutions of Hamiltonian systems on almost every
energy surfaces, \emph{Bol.\ Soc.\ Bras.\ Mat.}, \textbf{20} (1990),
49--58.


\bibitem[Ta1]{Ta1}
I.A. Taimanov,
Closed extremals on two-dimensional manifolds,
\emph{Russian Math.\ Surveys}, \textbf{47}(2) (1992), 163--211.


\bibitem[Ta2]{Ta2}
I.A. Taimanov,
Closed non-self-intersecting extremals for multivalued functionals,
\emph{Siberian Math.\ J.}, \textbf{33}(4) (1992), 686--692.


\bibitem[Vi]{Vi:fun}
C. Viterbo,
Functors and computations in Floer cohomology, I, \emph{Geom.\ Funct.\
Anal.}, \textbf{9} (1999), 985--1033.


\bibitem[We1]{We:orbits}
 A. Weinstein,
Normal modes for non-linear Hamiltonian systems, \emph{Invent.\
Math.}, \textbf{20} (1973), 377--410.


\bibitem[We2]{We84}
A. Weinstein,
$C^0$ perturbation theorems for symplectic fixed points and lagrangian
intersections, in \emph{S\'eminare sud-rhodanien de g\'eometri\'e. Travaux
en cours.}, Paris, Hermann, 1984, pp.\ 140--144.


\end{thebibliography}
\end{document}